\documentclass{amsart}

\usepackage{amsthm,amssymb}
\usepackage{amsmath,amscd}
\usepackage{amssymb}

\usepackage{pinlabel}
\usepackage{lscape}
\usepackage{latexsym}

\newtheorem{Thm}{Theorem}[section]
\newtheorem{Cor}[Thm]{Corollary}
\newtheorem{Prop}[Thm]{Proposition}
\newtheorem{Lem}[Thm]{Lemma}

\newtheorem*{thma}{Theorem A}
\newtheorem*{thmb}{Theorem B}
\newtheorem*{thmc}{Theorem C}

\theoremstyle{definition}
\newtheorem{Def}[Thm]{Definition}
\newtheorem{Ex}[Thm]{Example}

\theoremstyle{remark}

\numberwithin{equation}{section}

\newcommand{\Aut}{\operatorname{Aut}}

\newcommand{\Hom}{\operatorname{Hom}}
\newcommand{\Mor}{\operatorname{Mor}}

\newcommand{\Syl}{\operatorname{Syl}}
\newcommand{\Iso}{\operatorname{Iso}}
\newcommand{\Out}{\operatorname{Out}}

\newcommand{\Inn}{\operatorname{Inn}}
\newcommand{\Rep}{\operatorname{Rep}}

\newcommand{\res}{\operatorname{res}}
\renewcommand{\Gamma}{\varGamma}
\renewcommand{\epsilon}{\varepsilon}
\renewcommand{\bar}{\overline}
\renewcommand{\hat}{\widehat}
\renewcommand{\leq}{\leqslant}
\renewcommand{\geq}{\geqslant}

\newcommand{\K}{\mathcal{K} }
\newcommand{\I}{\mathcal{I} }

\newcommand{\F}{\mathcal{F}}

\newcommand{\G}{\mathcal{G}}
\newcommand{\T}{\mathcal{T}}

\newcommand{\E}{\mathcal{E}}
\newcommand{\C}{\mathcal{C}}
\renewcommand{\S}{\mathcal{S}}
\newcommand{\X}{\mathcal{X}}

\begin{document}


\title{Trees of fusion systems}
\thanks{Accepted for publication subject to revisions by Journal of Algebra, October 2013}

\author{Jason Semeraro}
\address{Institute of Mathematics, University of Oxford, 24-29 St Giles, Oxford}
\email{semeraro@maths.ox.ac.uk}
\maketitle




\begin{abstract}
We define a `tree of fusion systems' and give a sufficient condition for its completion to be saturated. We apply this result to enlarge an arbitrary fusion system by extending the automorphism groups of certain of its subgroups. 
\end{abstract}

Saturated fusion systems have come into prominence during the course of the last two decades and may be viewed as a convenient language in which to study some types of algebraic objects at a particular prime. The original idea is due to Puig in \cite{Puig}, although our notation and terminology more closely follows that of Broto, Levi and Oliver in \cite{BLO1}. The area is now well established and has attracted wide interest from researchers in group theory, topology and representation theory. Formally, a fusion system on a finite $p$-group $S$ is a category $\F$ whose objects consist of all subgroups of $S$ and whose morphisms a certain group monomorphisms between objects. Saturation is an additional property, which is satisfied by many `naturally occurring' fusion systems, including those induced by finite groups.

Recall that a tree of groups consists of a tree $\T$ with an assignment of groups to vertices and edges, and monomorphisms from `edge groups' to incident `vertex groups'. The completion of a tree of groups is the free product of all vertex groups modulo relations determined by the monomorphisms. The theory of trees of groups is a special case of Bass--Serre theory and this paper is an attempt to extend this theory to fusion systems (over a fixed prime $p$)  by attaching fusion systems to vertices and edges of some fixed tree, and injective morphisms from edge to vertex fusion systems. As in the case for groups, the completion of a `tree of fusion systems' is a colimit for the natural diagram, and we find a condition which renders this the fusion system generated by each of those attached to a vertex. 

One naturally obtains a tree of fusion systems from a tree of finite groups by replacing each group in the tree by its fusion systems at the prime $p$.  Two questions arise at this point:
\begin{itemize}
\item[(1)] Is the completion of this tree of fusion systems the fusion system of the completion of the underlying tree of groups?
\item[(2)] Is it possible to prove that the completion is saturated `fusion theoretically,' i.e. without reference to the groups themselves?
\end{itemize}
We answer (1) in the affirmative essentially by applying a straightforward lemma of Robinson concerning conjugacy relations in completions of trees of groups:

\begin{thma}
Let $(\T,\G)$ be a tree of finite groups and write $\G_\T$ for the completion of $(\T,\G)$. Let $(\T,\F,\S)$ be a tree of fusion systems induced by $(\T,\G)$ which satisfies $(H)$ so that there exists a completion $\F_\T$ for $(\T,\F,\S)$. The following hold:
\begin{itemize}
\item[(a)] $\S(v_*)$ is a Sylow $p$-subgroup of $\G_\T$.
\item[(b)] $\F_{\S(v_*)}(\G_\T)= \F_\T$.
\end{itemize}
In particular, $\F_\T$ is independent of the choice of tree of fusion systems $(\T,\F,\S)$ induced by $(\T,\G)$.
\end{thma} 
In this theorem, $(H)$ is a condition on $(\T,\F,\S)$ which forces the existence of a completion $\F_\T$ on the $p$-group $\S(v_*)$ associated to a fixed vertex $v_*$ of the underlying tree $\T$. To answer question (2) above, we require the introduction of some more terminology. Alperin's Theorem asserts that conjugacy in a saturated fusion system $\F$ on $S$ is determined by the `$\F$-essential' subgroups and $S$. Conversely a deep result of Puig asserts that whenever conjugacy in $\F$ is determined by the `$\F$-centric' subgroups, $\F$ is saturated if the saturation axioms hold between such subgroups. We exploit both of these facts in the proof of the following theorem:

\begin{thmb}
Let $(\T,\F,\S)$ be a tree of fusion systems which satisfies $(H)$ and assume that $\F(v)$ is saturated for each vertex $v$ of $\T.$ Write $S:=\S(v_*)$ and $\F_\T$ for the completion of $(\T,\F,\S)$. Assume that the following hold for each $P \leq S$:
\begin{itemize}
\item[(a)] If $P$ is $\F_\T$-conjugate to an $\F(v)$-essential subgroup or $P=\S(v)$ then $P$ is $\F_\T$-centric.
\item[(b)] If $P$ is $\F_\T$-centric then $\Rep_{\F_\T}(P,\F)$ is a tree. 
\end{itemize}
Then $\F_\T$ is a saturated fusion system on $S$.
\end{thmb}
Theorem B in the case where $(\T,\F,\S)$ is induced by a tree of groups $(\T,\G)$ is \cite[Theorem 4.2]{BLO4}, and can be deduced from Theorem B by applying Theorem A. The novelty of our approach is the introduction of the graph $\Rep_{\F_\T}(P,\F)$ called the `$P$-orbit graph' (defined for each $P \leq \S(v_*)$) which gives detailed information about the way in which $P$ `acts' on $\F_\T$. Condition (a) ensures that conjugacy in $\F_\T$ is determined by the $\F_\T$-centric subgroups and (b) ensures that the saturation axioms hold between such subgroups.

Theorem B is useful in determining fusion systems over specific (families of) $p$-groups. For example, in \cite{O5}, Oliver applies \cite[Theorem 5.1]{BLO4} (which follows from \cite[Theorem 4.2]{BLO4}) to prove that saturated fusion systems over $p$-groups with abelian subgroup of index $p$ are uniquely determined by their essential subgroups.  Our next result is a generalisation of \cite[Theorem 5.1]{BLO4} to arbitrary fusion systems.

\begin{thmc}
Let $\F_0$ be a saturated fusion system on a finite $p$-group $S$. For $1 \leq i \leq m$, let $Q_i \leq S$ be fully $\F_0$-normalised subgroups with $Q_i\varphi \nleq Q_j$ for each $\varphi \in \Hom_{\F_0}(Q_i,S)$ and $i \neq j$. Set $K_i:=\Out_{\F_0}(Q_i)$ and choose $\Delta_i \leq \Out(Q_i)$ so that $K_i$ is a strongly $p$-embedded subgroup of $\Delta_i$. Write $$\F= \langle \{\Hom_{\F_0}(P,S) \mid P \leq S\} \cup \{\Delta_i \mid 1 \leq i \leq m \} \rangle_S.$$ Assume further that for each $1 \leq i \leq m$,
\begin{itemize}
\item[(a)] $Q_i$ is $\F_0$-centric (hence $\F$-centric) and minimal (under inclusion) amongst all $\F$-centric subgroups; and
\item[(b)] no proper subgroup of $Q_i$ is $\F_0$-essential. 
\end{itemize}
Then $\F$ is saturated.
\end{thmc}

Recall that a subgroup $H$ of a finite group $G$ is \textit{strongly $p$-embedded} if $H < G$, $H$ contains a Sylow $p$-subgroup of $G$ and $H \cap H^g$ is a $p'$-group for each $g \in G \backslash H$. The proof Theorem C is similar in structure to that of \cite[Theorem 5.1]{BLO4}, but uses less group theory.

The paper is structured as follows. In Section \ref{section1}, we introduce fusion systems, saturation, Alperin's theorem and $\F$-normaliser subsystems. We then introduce the main objects of study - trees of fusion systems -  in Section \ref{treesfus}, along with their completions and $P$-orbit graphs. In Section \ref{orbfus}, we describe the relationship between trees of groups and trees of fusion systems and prove Theorem A. Theorem B will be proved in Section \ref{compfus} before being applied in Section \ref{extconst} to prove Theorem C.

Finally, we make some important remarks regarding notation. Elements of a finite group $G$ typically act on the right. We reserve the superscript notation $x^g$ for image of the conjugation action $g^{-1}xg$ of $G$ on itself for each $x,g \in G$. Similarly if $H$ is a subgroup of $G$ we write $H^g$ for the group $\{x^g \mid x \in H\}$. We will also frequently write $c_g$ for the group homomorphism induced by conjugation by $g$, stating the domain where appropriate. Note that group homomorphisms will always act on the right and composition read from left to right.

\section{Background}\label{section1}
\subsection{Saturated Fusion Systems}
We begin with a precise definition of what is meant by a fusion system. For any group $S$ and $P,Q \leq S$ we write Hom$_S(P,Q)$ for the set of homomorphisms from $P$ to $Q$ induced by conjugation by elements of $S$ and Inj$(P,Q)$ for the set of monomorphisms from $P$ to $Q$.

\begin{Def}\label{fus}
Let $S$ be a finite $p$-group. A \textit{fusion system on $S$} is a category $\F$ where Ob$(\F):=\{P \mid P \leq S \}$, and where for each $P,Q \in$ Ob$(\F)$
\begin{itemize}
\item[(a)] Hom$_S(P,Q) \subseteq$ Hom$_\F(P,Q) \subseteq$ Inj$(P,Q)$; and
\item[(b)] each $\varphi \in$ Hom$_\F(P,Q)$ factorises as an $\F$-isomorphism $\begin{CD}
P @>>> P\varphi \\
\end{CD}$ followed by an \textit{inclusion} $\begin{CD}
\iota_{P\varphi}^Q:P\varphi @>>> Q. \\
\end{CD}$
\end{itemize} 
\end{Def}

Given an arbitrary collection $X_{P,Q} \subseteq$ Inj$(P,Q)$ of morphisms which contains Hom$_S(P,Q)$ for each $P,Q \leq S$, we can always construct a fusion system $\F$ on $S$ with $X_{P,Q} \subseteq$ Hom$_\F(P,Q)$ and where Hom$_\F(P,Q)$ $\backslash$ $X_{P,Q}$ only consists of $\F$-isomorphisms. In particular, $\F$ is minimal (with the respect to the number of morphisms) amongst all fusion systems $\G$ on $S$ with the property that $X_{P,Q} \subseteq$ Hom$_\G(P,Q)$ for each $P,Q \leq S$. Next we define morphisms between fusion systems.

\begin{Def}
Let $\F$ and $\E$ be fusion systems on finite $p$-groups $S$ and $T$ respectively. $\varphi \in \Hom(S,T)$ is a \textit{morphism} from $\F$ to $\E$ if for each $P,R \leq S$ and each $\alpha \in \Hom_\F(P,R)$, there exists $\beta \in \mbox{Hom}_\E(P\varphi,R\varphi)$ such that $$\alpha \circ \varphi|_R=\varphi|_P \circ \beta.$$ Hence $\varphi$ induces a functor from $\F$ to $\E$.
\end{Def}

In particular, fusion systems form a category which we denote by $\mathfrak{Fus}$. The following definition collects some important notions to which we will constantly refer, all of which are needed to define saturation. The reader is referred to \cite[Section I.2]{AKO} and \cite[Section 1.5]{CR} for a more thorough introduction to these ideas.

\begin{Def}
Let $\F$ be a fusion system on a finite $p$-group $S$ and let $P,Q \leq S$.
\begin{itemize}
\item[(a)] Iso$_\F(P,Q)$ denotes the set of $\F$-isomorphisms from $P$ to $Q$ and Aut$_\F(P) $:= Iso$_\F(P,P)$. 
\item[(b)] $P$ and $Q$ are \textit{$\F$-conjugate} whenever Iso$_\F(P,Q)  \neq \emptyset$ and $P^\F$ denotes the set of all $\F$-conjugates of $P$.

\item[(c)] $P$ is \textit{fully $\F$-normalised} respectively \textit{fully $\F$-centralised} if for each $R \in P^\F$, the inequality $$|N_S(P)| \geq |N_S(R)| \mbox{ respectively } |C_S(P)| \geq |C_S(R)|$$ holds. 
\item[(d)] $P$ is \textit{fully $\F$-automised} if Aut$_S(P) \in$ Syl$_p($Aut$_\F(P)).$
\item[(e)] For each $\varphi \in$ Hom$_\F(P,S)$, $$N_\varphi=N_\varphi^\F:=\{g \in N_S(P) \mid \varphi^{-1} \circ c_g \circ \varphi \in \mbox{Aut}_S(P\varphi) \},$$ where $c_g \in$ Aut$_S(P)$ is the automorphism induced by conjugation by $g$.
\end{itemize}
\end{Def}
There are a number of equivalent ways to define saturation and we refer the reader to  \cite[Section I.9]{AKO} or \cite[Section 4.3]{CR} for a comparison of these. The definition we choose is listed among the definitions in these references and also appears in \cite[Definition 1.2]{BLO1}.

\begin{Def}\label{sat}
Let $\F$ be a fusion system on a finite $p$-group $S$. We say that $\F$ is \textit{saturated} if the following hold:
\begin{itemize}
\item[(a)] Whenever $P \leq S$ is fully $\F$-normalised, it is fully $\F$-centralised and fully $\F$-automised.
\item[(b)] For all $P \leq S$ and $\varphi \in$ Hom$_\F(P,S)$ such that $P\varphi$ is fully $\F$-centralised, there is $\bar{\varphi} \in $ Hom$_\F(N_\varphi, S)$ such that $\bar{\varphi}|_P=\varphi$.
\end{itemize}
\end{Def}

We observe that a natural class of examples of saturated fusion systems is provided by groups. A finite $p$-subgroup $S$ of a group $G$ is a \textit{Sylow $p$-subgroup of $G$} if every finite $p$-subgroup of $G$ is $G$-conjugate to a subgroup of $S$. Let $\F_S(G)$ denote the fusion system on $S$ where for each $P,Q \leq S$, Hom$_{\F_S(G)}(P,Q):=$ Hom$_G(P,Q).$

\begin{Thm}\label{fsgsat}
Let $G$ be a finite group and $S$ be a Sylow $p$-subgroup of $G$. The fusion system $\F_S(G)$ is saturated.
\end{Thm}

\subsection{Alperin's Theorem}
We now concern ourselves with `generation' of saturated fusion systems, starting with some more definitions:
\begin{Def}
Let $\F$ be a fusion system on a finite $p$-group $S$ and let $P \leq S.$
\begin{itemize}
\item[(a)] $P$ is \textit{$S$-centric} if $C_S(P)=Z(P)$ and $P$ is \textit{$\F$-centric} if $Q$ is $S$-centric for each $Q \in P^\F.$ 
\item[(b)] Write $\Out_\F(P):=\Aut_\F(P)/\Inn(P)$. $P$ is \textit{$\F$-essential} if $P$ is $\F$-centric and $\Out_\F(P)$ contains a strongly $p$-embedded subgroup. 
\item[(c)] $P$ is \textit{$\F$-radical} if $O_p(\Out_\F(P))=1$.
\end{itemize}
\end{Def}
The following lemma concerning $\F$-centric subgroups is extremely useful.

\begin{Lem}\label{centlem}
Let $\F$ be a fusion system on a finite $p$-group $S$. The following hold for each $P,Q \leq S$:
\begin{itemize}
\item[(a)] If $P$ is fully $\F$-centralised then $PC_S(P)$ is $\F$-centric.
\item[(b)] If $P \leq Q$ and $P$ is $\F$-centric, then $Q$ is $\F$-centric.
\end{itemize}
\end{Lem}

\begin{proof}
Writing $R=PC_S(P)$, we see immediately that $C_S(R) \leq R$. If $\varphi \in$ Hom$_\F(R,S)$ then $R\varphi \leq P\varphi C_S(P\varphi)$ and since $P$ is fully $\F$-centralised, necessarily $R\varphi = P\varphi C_S(P\varphi).$ Now $C_S(R\varphi) \leq R\varphi$ which proves (a). Part (b) is trivial. 
\end{proof}
We introduce some notation which makes precise the notion of a `generating' set in our context.

\begin{Def}
Let $S$ be a finite $p$-group and let $\C$ be a collection of injective maps between subgroups of $S$. Denote by $\langle \C \rangle_S$ the smallest fusion system on $S$ containing all maps which lie in $\C$. If $\F$ is a fusion system on $S$ and $\F=\langle \C \rangle_S$ then the set $\C$ is said to \textit{generate} $\F$.
\end{Def}
Observe that whenever $\C$ generates $\F$, each morphism in $\F$ can be written as a composite of restrictions of elements of $\C$. 
\begin{Def}
A set of subgroups $\X$ of a finite $p$-group $S$ to be a \textit{conjugation family} for a fusion system $\F$ on $S$ if $\langle\Aut_\F(P) \mid P \in \X \rangle_S=\F$.
\end{Def}
We may now state Alperin's theorem for saturated fusion systems which provides a very useful conjugation family for any saturated fusion system $\F$.
\begin{Thm}\label{alpthm}
Let $\F$ be a saturated fusion system on a finite $p$-group $S$. $$\{S\} \cup \{P \mid \mbox{$P$ is fully $\F$-normalised and $\F$-essential} \}$$ is a conjugation family for $\F$.
\end{Thm}

\begin{proof}
See \cite[Theorem I.3.5]{AKO}.
\end{proof}

Proving that a fusion system $\F$ is saturated can often be a difficult task, since there may be many subgroups for which the saturation axioms must be checked. Fortunately, this job is occasionally made easier when an $\F$-conjugation family is known to exist.

\begin{Thm}\label{alpthmconv}
Let $\F$ be a fusion system on a finite $p$-group $S$. If $\F$-centric subgroups form a conjugation family then $\F$ is saturated if (a) and (b) in Definition \ref{fus} hold for all such subgroups.
\end{Thm}

\begin{proof}
See \cite[Theorem 3.8]{Puig}.
\end{proof}
We think of Theorem \ref{alpthmconv} as a partial converse to Theorem \ref{alpthm} and it is a fundamental tool in our argument to prove Theorem B.

\subsection{$\F$-normalisers and Constrained Fusion Systems}
This section introduces two concepts which will be used in the proof of Theorem C. We begin with the analogue for fusion systems of the ordinary normaliser of a subgroup of a finite group.
\begin{Def}
Let $\F$ be a fusion system on a finite $p$-group $S$ and let $Q \leq S$. The \textit{$\F$-normaliser of $Q$}, $N_\F(Q)$ is the fusion system on $N_S(Q)$ where for each $P,R \leq N_S(Q)$, Hom$_{N_\F(Q)}(P,R)$ is the set $$\{\varphi \in \mbox{Hom}_\F(P,R) \mid \mbox{ $\exists$ } \bar{\varphi} \in \mbox{Hom}_\F(PQ,RQ) \mbox{ s.t. } \bar{\varphi}|_P=\varphi \mbox{ and } \bar{\varphi}|_Q \in \Aut(Q)\}.$$ 
\end{Def}

\begin{Thm}\label{nfkq}
Let $\F$ be a saturated fusion system on a finite $p$-group $S$ and let $Q \leq S$. If $Q$ is fully $\F$-normalised then $N_\F(Q)$ is saturated.
\end{Thm}

\begin{proof}
See \cite[Theorem I.5.5]{AKO}.
\end{proof}
The notion of an $\F$-normaliser naturally gives rise to the notion of a normal subgroup of a fusion system as follows.
\begin{Def}\label{normdef}
Let $\F$ be a fusion system on a finite $p$-group $S$ and let $Q \leq S$. $Q$ is \textit{normal} in $\F$ if $N_\F(Q)=\F$. Write $O_p(\F)$ for the maximal normal subgroup of $\F.$
\end{Def}
The following lemma provides a characterisation of $O_p(\F)$ for a saturated fusion system $\F$ in terms of its fully $\F$-normalised, $\F$-essential subgroups.
\begin{Lem}\label{esscont}
Let $\F$ be a saturated fusion system on a finite $p$-group $S$ and let $Q \leq S$. If $Q$ is normal in $\F$ then $Q \leq R$ for each $\F$-centric $\F$-radical subgroup $R$ of $S$. Conversely, if $Q \leq R$ for each fully $\F$-normalised, $\F$-essential subgroup $R$ of $S$, then $Q$ is normal in $\F$.
\end{Lem}

\begin{proof}
See \cite[Theorem 4.61]{CR}.
\end{proof}

The notion of a normal subgroup can also be applied to define the analogue for fusion systems of a $p$-constrained finite group. Recall that a finite group $G$ with $O_{p'}(G)=1$ is \textit{$p$-constrained} if there exists a normal subgroup $R \unlhd G$ with $C_G(R) \leq R$.
\begin{Def}
A saturated fusion system $\F$ on a finite $p$-group $S$ is \textit{constrained} if there exists an $\F$-centric subgroup of $S$ which is normal in $\F$. 
\end{Def}

The final result of this section asserts that every constrained fusion system arises as the fusion system of a $p$-constrained group.
\begin{Thm}\label{const}
Let $\F$ be a fusion system on a finite $p$-group $S$ and suppose that there exists an $\F$-centric subgroup $R$ of $S$ which is normal in $\F$. There exists a unique finite group $G$ with $S \in \Syl_p(G)$ with the properties that $$\F=\F_S(G), \mbox{ } R \unlhd G, \mbox{ } \Out_G(R) \cong G/R, \mbox{ and } O_{p'}(G)=1.$$
\end{Thm} 

\begin{proof}
See, for example \cite[Theorem I.5.10]{AKO}.
\end{proof}

\section{Trees of Fusion Systems}\label{treesfus}
In this section, we will carefully define what we mean by a tree of fusion systems and the completion of such an object. We then find a natural condition which ensures that the completion of a tree of fusion systems exists. We warn the reader that from now on the symbol $`\F$' will frequently be used to denote a functor with values in the category of fusion systems, rather than just a single fusion system.
\subsection{Trees of Groups}
We begin by introducing some notation. If $\T$ is a simple, undirected graph, write $V(\T)$ and $E(\T)$ for the sets of vertices and edges of $\T$ respectively. Each edge $e \in E(\T)$ is regarded as an unordered pair of vertices, $(v,w)$ say, and $v$ and $w$ are said to be \textit{incident} on $e$. $\T$ gives rise to a category (also called $\T$) where Ob($\T$) is the disjoint union of the sets $V(\T)$ and $E(\T)$ and where for each edge $(v,w) \in E(\T)$ there exists a pair of morphisms $$
\begin{CD}
e @>>> v\\
@VVV \\
w 
\end{CD}$$ in $\T$. We denote the unique morphism in Hom$_\T(e,v)$ by $f_{ev}$ and write $\mathfrak{Grp}$ for the category of groups and group homomorphisms.
\begin{Def}
A \textit{tree of groups} is a pair $(\T,\G)$ where $\T$ is a tree and $\G$ is a functor from $\T$ (regarded as a category) to $\mathfrak{Grp}$ which sends $f_{ev}$ to group monomorphism from $\G(e)$ to $\G(v)$. The \textit{completion} $\G_\T$ of $(\T,\G)$ is the group $ \underrightarrow{\mbox{colim}}_{
  \substack{
  \T}} \G.$ 
\end{Def} 

\begin{Lem}
Each tree of groups $(\T,\G)$ has a completion which is unique up to group isomorphism.
\end{Lem}

Let $(\T,\G)$ be a tree of groups with completion $G:=\G_\T$. Define $G/\G(-)$ to be the functor from $\T$ to $\mathfrak{Set}$ which sends each $v$ and $e \in$ Ob$(\T)$ respectively to the sets of left cosets $G/\G(v)$ and $G/\G(e)$ and which sends $f_{ev} \in$ Hom$_\T(e,v)$ to the map from $G/\G(e)$ to $G/\G(v)$ given by sending left cosets $g\G(e)$ to $g\G(v)$.  Define the \textit{orbit graph} $\tilde{\T}$ to be the space $$ \underrightarrow{\mbox{hocolim}}_{
  \substack{
  \T}} G/\G(-).$$ 
Equivalently, we may think of $\tilde{\T}$ as being a graph whose vertices and edges are labelled by the sets $\{g\G(v) \mid g \in G, v \in V(\T)\}$ and $\{i\G(e) \mid i \in G, e \in E(\T)\}$ respectively and where two vertices $g\G(v)$ and $h\G(w)$ are connected via an edge $i\G(e)$ if and only if $i\G(v)=g\G(v)$ and $i\G(w)=h\G(w)$. This is obviously a graph on which $G$ acts by left multiplication. Denote by $\tilde{\T}/G$ the graph with vertex set given by the set of orbits of $V(\tilde{\T})$ under the action of $G$ on $V(\tilde{\T})$ and likewise for the edges. We have the following theorem:

\begin{Thm}\label{fundbass}
Let $(\T,\G)$ be a tree of groups with completion $G:=\G_\T$. Then $\tilde{\T}$ is a tree and $\tilde{\T}/G \simeq \T.$
\end{Thm}

\begin{proof}
See \cite[I.4.5]{Serre}
\end{proof}

\subsection{Trees of Fusion Systems}
Suppose that $\F$ and $\E$ are fusion systems on finite $p$-groups $S$ and $T$ respectively. A morphism $\alpha \in$ Hom$(S,T)$ from $\F$ to $\E$ is \textit{injective} if it induces an injective map $$\begin{CD}
\mbox{Hom}_\F(P,S) @>>> \mbox{Hom}_\E(P\alpha,T)\\
\end{CD}$$ for each $P \leq S$.

\begin{Def}
A \textit{tree of $p$-fusion systems} is a triple $(\T,\F,\S)$ where $(\T,\S)$ is a tree of finite $p$-groups and $\F$ is a functor from $\T$ to $\mathfrak{Fus}$ such that the following hold:
\begin{itemize}
\item[(a)] $\F(v)$ is a fusion system on $\S(v)$ and $\F(e)$ is a fusion system on $\S(e)$, and
\item[(b)] $\F$ sends $f_{ev} \in$ Hom$_\T(e,v)$ to an injective morphism from $\F(e)$ to $\F(v).$
\end{itemize}
\end{Def}

\begin{Def}
Let $(\T,\F,\S)$ be a tree of $p$-fusion systems. The \textit{completion} of $(\T,\F,\S)$ is a colimit for $\F$.
\end{Def} 

Of course, we need conditions on $(\T,\F,\S)$ which imply that a colimit for $\F$ exists, since this is no longer guaranteed as it was in the category of groups. Indeed, any colimit for $\F$ must be a fusion system on the completion $S_\T$ of $(\T,\S)$ and this may not be a $p$-group\footnote{Consider, for example the amalgam $C_2 * C_2 \cong D_{\infty}$ when $p=2$.}. The following is a very simple (and natural) condition to impose: \newline
\newline
\textbf{Hypothesis} \textit{$(H)$: There exists a vertex $v_* \in V(\T)$ with the property that whenever $v \in V(\T)$, $\S(e) \cong \S(v)$ where $e$ is the edge incident to $v$ in the unique minimal path from $v$ to $v_*$.} \newline
\newline
We will say that a tree of $p$-fusion systems $(\T,\F,\S)$ satisfies $(H)$ if $(\T,\S)$ satisfies Hypothesis $(H)$. Let $(\T,\F,\S)$ be a tree of $p$-fusion systems which satisfies $(H)$ and write $S:=\S(v_*)$. It is clear that $\S_\T$ is a finite group isomorphic to $\S(v_*)$, so that we may view $\S(e)$ and $\S(v)$ as subgroups of $\S(v_*)$ by identifying them with their images in the completion. Also, $\S(v) \cap \S(w)=\S(e)$ in $\S(v_*)$ whenever $v$ and $w$ are vertices of $\T$ incident on $e$. Furthermore, when $v$ is a vertex incident to an edge $e$, this identification allows us to embed each fusion system $\F(e)$ as a subsystem of $\F(v)$ by identifying each morphism $\alpha \in$ Hom$_{\F(e)}(P,\S(e))$ with its image under the functor from $\F(e)$ to $\F(v)$ induced by $f_{ev} \in$ Hom$_\T(e,v).$ 

\begin{Lem}\label{compft}
Let $(\T,\F,\S)$ be a tree of $p$-fusion systems which satisfies $(H)$, set $v_0:=v_*$ and write $S:=\S(v_0)=\S_\T$. Define
$$\F_\T:=\langle \Hom_{\F(v)}(P,\S(v)) \mid P \leq \S(v),v \in V(\T) \rangle_S,$$ the fusion system generated by the $\F(v)$ for each $v \in V(\T)$. $\F_\T$ is a colimit for $\F$ and each $\alpha \in \Hom_{\F_\T}(P,S)$ may be written as a composite
 $$\begin{CD}
P=P_0 @>\alpha_0>> P_1 @>\alpha_1>> P_2 @>\alpha_2>> \cdots @>\alpha_{n-1}>> P_n=P\alpha\\
\end{CD}$$
 where for $0 \leq i \leq n-1$, $P_i \leq \S(v_{i-1}) \cap \S(v_i)$, $\alpha_i \in \Mor(\F(v_i))$ for some $v_i \in V(\T)$ and  $(v_i,v_{i+1})$ is an edge in $\T$.
\end{Lem}

\begin{proof}
The fact that $\F_\T$ is a colimit for $\F$ follows immediately from the fact that it is unique (up to an isomorphism of categories) amongst all fusion systems on $S$ which contain $\F(v)$ for each $v \in V(\T)$\footnote{More information concerning this characterisation of $\F_\T$ is provided by the remarks which follow Definition \ref{fus}}. To see the second statement, clearly any such composite of morphisms lies in $\F_\T$. Conversely let $$\begin{CD}
P=P_0 @>\alpha_0>> P_1 @>\alpha_1>> P_2 @>\alpha_2>> \cdots @>\alpha_{n-1}>> P_n=P\alpha\\
\end{CD}$$ be a representation of $\alpha \in$ Hom$_{\F_\T}(P,S)$ where $\alpha_i \in$ Mor$(\F(v_i))$ for $0 \leq i \leq n-1$.  If $(v_{i-1},v_i)$ is not an edge then let $\eta$ be a minimal path in $\T$ from $v_{i-1}$ to $v_i$. Since $(\T,S)$ satisfies $(H)$, $P_i$ is contained in $\S(w)$ for each vertex $w \in \eta$ so that by inserting identity morphisms, the above sequence of morphisms can be refined so that $(v_{i-1},v_i)$ is an edge for each $i$. This completes the proof of the lemma.
\end{proof}
One observes that by Lemma \ref{compft}, the competion $\F_\T$ of a tree of $p$-fusion systems $(\T,\F,\S)$ which satisfies $(H)$ is independent of where $\F$ sends the edges $e$ of $\T$. We will make heavy use of this fact later in the proof of Theorem B.

\subsection{The $P$-orbit Graph}
We now turn to the definition of a certain graph constructed from a tree of $p$-fusion systems (now simply referred to as a tree of fusion systems) and an arbitrary finite $p$-group $P$. Our discussion culminates in the proof of an important result, Proposition \ref{phipft}, which will allow us to describe morphisms in the completion combinatorially.

We begin by introducing some notation. Let $\K$ be a fusion system on a finite $p$-group $S$ and let $P$ be any other finite $p$-group. We define an equivalence relation $\sim$ on the set of homomorphisms $\Hom(P,S)$ as follows. For $\alpha, \beta \in \Hom(P,S)$ define $\alpha \sim \beta$ if and only if there is some $\gamma \in  \Iso_\K(P\alpha,P\beta)$ such that $\alpha \circ \gamma = \beta$. The fact that $\sim$ is an equivalence relation follows from the axioms for a fusion system. Write $$\mbox{Rep}(P,\K):= \mbox{Hom}(P,S)/\sim$$ for the set of all equivalence classes, and for each $\alpha \in $ Hom$(P,S)$ let $[\alpha]_\K$ denote the class of $\alpha$ in Rep$(P,\K)$. If $\hat{\K}$ is a fusion system containing $\K$, set Rep$_{\hat{\K}}(P,\K):=$ Hom$_{\hat{\K}}(P,S)/\sim$. 

Let $(\T,\F,\S)$ be a tree of fusion systems, and $P$ be a finite $p$-group. Define Rep$(P,\F(-))$ to be the functor from $\T$ to $\mathfrak{Set}$ which sends vertices $v$ and edges $e$ of $\T$ respectively to Rep$(P,\F(v))$ and Rep$(P,\F(e))$ and which sends $f_{ev} \in$ Hom$_\T(e,v)$ to the map from Rep$(P,\F(e))$ to Rep$(P,\F(v))$ given by $$[\alpha]_{\F(e)} \longmapsto [\alpha \circ \iota_{\S(e)}^{\S(v)}]_{\F(v)}.$$ Observe that this mapping is independent of the choice of $\alpha$. Using this definition we introduce the following space.
\begin{Def}
Let $P$ be a finite $p$-group and $(\T,\F,\S)$ be a tree of fusion systems. The \textit{$P$-orbit graph}, Rep$(P,\F)$ is the homotopy colimit $$ \underrightarrow{\mbox{hocolim}}_{
  \substack{
  \T}} \Rep(P,\F(-)).$$
\end{Def}
Since there are no $n$-simplices in $ \underrightarrow{\mbox{hocolim}}_{
  \substack{
  \T}} \Rep(P,\F(-))$ for $n \geq 2$, Rep$(P,\F)$ can be described as the geometric realisation of a graph as follows.

\begin{Lem}\label{repfgraph}
Let $P$ be an arbitrary $p$-group and $(\T,\F,\S)$ be a tree of fusion systems. Then $R:=$ Rep$(P,\F)$ may be regarded as a graph with
$$V(R) = \bigcup_{v \in V(\T)} \Rep(P,\F(v)) \mbox{ and } E(R) = \bigcup_{e \in E(\T)}  \Rep(P,\F(e)),$$ where $[\alpha]_{\F(v)}$ and $[\beta]_{\F(w)}$ are connected via an edge $[\gamma]_{\F(e)}$ if and only if $v$ and $w$ are both incident on $e$ in $\T$ and the identities $$[\gamma \circ \iota_{\S(e)}^{\S(v)}]_{\F(v)}=[\alpha]_{\F(v)} \mbox{ and } [\gamma \circ \iota_{\S(e)}^{\S(v)}]_{\F(w)}=[\beta]_{\F(w)}$$ hold. 
\end{Lem}

\begin{proof}
This follows immediately from the definition of  the homotopy colimit.
\end{proof}

If $(\T,\F,\S)$ is a tree of fusion systems which satisfies $(H)$, then by Lemma \ref{compft}, $(\T,\F,\S)$ has a completion which we denote by $\F_\T.$ Let Rep$_{\F_\T}(P,\F(-))$ be the subfunctor of Rep$(P,\F(-))$ which sends vertices $v$ and edges $e$ of $\T$ respectively to Rep$_{\F_\T}(P,\F(v))$ and Rep$_{\F_\T}(P,\F(e))$ and which sends $f_{ev}$ to the map from Rep$_{\F_\T}(P,\F(e))$ to Rep$_{\F_\T}(P,\F(v))$ given by $$[\alpha]_{\F(e)} \longmapsto [\alpha \circ \iota_{\S(e)}^{\S(v)}]_{\F(v)}.$$ (Note that this map is well-defined since $\F_\T$ is closed under composition with inclusion morphisms). Let $$\mbox{Rep}_{\F_\T}(P,\F):= \underrightarrow{\mbox{hocolim}}_{
  \substack{
  \T}} \Rep_{\F_\T}(P,\F(-)).$$
We observe that the obvious analogue of Lemma \ref{repfgraph} holds for Rep$_{\F_\T}(P,\F)$, and that (for this reason) Rep$_{\F_\T}(P,\F)$ may be embedded as a subgraph of Rep$(P,\F)$ in the obvious way. We end this section with an important result which provides us with a precise description of this embedding.

\begin{Prop}\label{phipft}
Let $(\T,\F,\S)$ be a tree of fusion systems which satisfies $(H)$ and let $\F_\T$ be its completion. Write $S:=\S(v_*)$ and fix $P \leq S$. The following hold:
\begin{itemize}
\item[(a)] The connected component of the vertex $[\iota_P^S]_{\F(v_*)}$ in $\Rep(P,\F)$ is isomorphic to Rep$_{\F_\T}(P,\F)$.
\item[(b)] The natural map $$\begin{CD}
\Phi_P: \pi_0(\Rep(P,\F)) @>>> \Rep(P,\F_\T)\\
\end{CD}$$ which sends the connected component of a vertex $[\alpha]_{\F(v)}$ in $\Rep(P,\F)$ to $[\alpha \circ \iota_{\S(v)}^S]_{\F_\T}$ is a bijection.
\end{itemize}
\end{Prop}
\begin{proof}
Since $\T$ has finitely many vertices and $\S(v)$ is finite for each $v \in V(\T)$, the graph Rep$(P,\F)$ contains finitely many vertices and edges. We may identify $\F(v)$ and $\F(e)$ with their images in $\F_\T$ by Lemma \ref{compft}. Set $v_0:=v_*$ and $\alpha_0:= \iota_P^{\S(v_0)}$ and let $$[\alpha_0]_{\F(v_0)}, [\alpha_1]_{\F(v_1)}, \ldots, [\alpha_n]_{\F(v_n)}$$ be a path in Rep$(P,\F)$ and assume that $[\beta_i]_{\F(e_i)}$ is an edge from $[\alpha_{i-1}]_{\F(v_0)}$ to $[\alpha_i]_{\F(v_1)}$ for each $1 \leq i \leq n.$
 We need to show that each vertex $[\alpha_i]_{\F(v_i)}$ lies in Rep$_{\F_\T}(P,\F)$. Clearly $[\alpha_0]_{\F(v_0)} \in$ Rep$_{\F_\T}(P,\F)$. Assume that $n \geq 1$, and that $[\alpha_i]_{\F(v_i)} \in$ Rep$_{\F_\T}(P,\F)$ for some $i < n$. If $[\beta_{i+1}]_{\F(e_{i+1})}$ is an edge from 
 $[\alpha_i]_{\F(v_i)}$ to $[\alpha_{i+1}]_{\F(v_{i+1})}$ then there exist maps $$\gamma \in \mbox{Hom}_{\F(v_i)}(P\alpha_i,P\beta_{i+1}) \mbox{ and } \delta \in \mbox{Hom}_{\F(v_{i+1})}(P\beta_{i+1},P\alpha_{i+1})$$ such that $\alpha_i \circ \gamma = \beta_{i+1}$ and $\beta_{i+1} \circ \delta=\alpha_{i+1}$. Hence $\gamma \circ \delta \in$ Hom$_{\F_\T}(P\alpha_i,P\alpha_{i+1}) $ and $[\alpha_{i+1}]_{\F(v_{i+1})} \in$ Rep$_{\F_\T}(P,\F)$ and by induction, $[\alpha_i]_{\F(v_i)} \in$ Rep$_{\F_\T}(P,\F)$ for all $0 \leq i \leq n.$

Conversely suppose that $\alpha \in$ Hom$_{\F_\T}(P,\S(v))$ and that $[\alpha]_{\F(v)}$ is a vertex in Rep$_{\F_\T}(P,\F)$. By the definition of $\F_\T$, there exists a path $v_0,v_1, \ldots, v_n=v$ in $\T$ and for $1 \leq i \leq n$, maps $\alpha_i \in$ Hom$_{\F(v_i)}(P\alpha_{i-1}, P\alpha_i)$ with $\alpha_0=\iota_P^S$ such that $\alpha=\alpha_0 \circ \cdots \circ \alpha_n.$ This implies that there exists a path in Rep$_{\F_\T}(P,\F)$, $$[\alpha_0]_{\F(v_0)}, [\alpha_0 \circ \alpha_1]_{\F(v_1)} \ldots, [\alpha_0 \circ \alpha_1 \circ \cdots \circ \alpha_{n-1}]_{\F(v_{n-1})},[\alpha]_{\F(v_n)},$$ and completes the proof of (a).

Next, we prove (b). It suffices to show that two vertices $[\alpha]_{\F(v)}$ and $[\beta]_{\F(w)}$ are connected in Rep$(P,\F)$ if and only if $[\alpha \circ \iota_{\S(v)}^S]_{\F_\T}$=$[\beta \circ \iota_{\S(w)}^S]_{\F_\T}$ in Rep$(P,\F_\T).$ It is enough to prove this when $[\alpha]_{\F(v)}$ and $[\beta]_{\F(w)}$ are connected via a single edge $[\gamma]_{\F(e)}$. 

We prove the `only if' direction first. Thus we suppose that $[\alpha]_{\F(v)}$ and $[\beta]_{\F(w)}$ are connected via an edge $[\gamma]_{\F(e)}$ in Rep$(P,\F)$ so that there exist morphisms $\gamma_1 \in$ Hom$_{\F(v)}(P\alpha,P\gamma)$ and $\gamma_2 \in$ Hom$_{\F(w)}(P\beta,P\gamma)$ with $$P\alpha\gamma_1=P\gamma=P\beta\gamma_2.$$ Then $\gamma_1\gamma_2^{-1} \in$ Hom$_{\F_\T}(P\alpha,P\beta)$ and $[\beta \circ \iota_{\S(w)}^S]=[\alpha\gamma_1\gamma_2^{-1} \circ \iota_{\S(w)}^S]=[\alpha \circ \iota_{\S(v)}^S] \in$ Rep$(P,\F_\T)$. Conversely, suppose that $[\alpha \circ \iota_{\S(v)}^S]_{\F_\T}=[\beta \circ \iota_{\S(w)}^S]_{\F_\T}=[\delta]_{\F_\T}$, for some $\delta \in$ Hom$_{\F_\T}(P,S)$. Then $[\delta^{-1}\alpha]_{\F(v)}$, $[\delta^{-1}\beta]_{\F(v)}$ are vertices in Rep$_{\F_\T}(P\delta,\F)$ and by part (a), this graph is connected so that there must exist a path joining  $[\delta^{-1}\alpha]_{\F(v)}$ to $[\delta^{-1}\beta]_{\F(v)}$ in Rep$_{\F_\T}(P\delta,\F)$. This is easily seen to be equivalent to the existence of a path joining $[\alpha]_{\F(v)}$ to $[\beta]_{\F(w)}$ in Rep$(P,\F)$, completing the proof of the lemma.
\end{proof}

\section{Trees of Group Fusion Systems}\label{orbfus}
In this section we will give a precise description of the relationship between trees of groups $(\T,\G)$ and trees of fusion systems $(\T,\F,\S)$ by considering what happens when the latter is induced by the former. It will turn out that both the completion and $P$-orbit graph of $(\T,\F,\S)$ have group-theoretic descriptions in this case, which will allow us to give an entirely group theoretic interpretation of Theorem B, in Section \ref{compfus}.

\subsection{The Completion}
We start by giving a precise explanation of the word `induced' above. The following lemma is a trivial consequence of Sylow's Theorem.

\begin{Lem}\label{induce}
Let $(\T,\G)$ be a tree of finite groups. For any choice of Sylow $p$-subgroups $\S(v) \in \Syl_p(\G(v))$ and $\S(e) \in \Syl_p(\G(e))$, there exists a tree of finite $p$-groups $(\T,\S)$  and a tree of fusion systems $(\T,\F_{\S(-)}(\G(-)),\S(-))$ \mbox{induced by} $(\T,\G)$. 
\end{Lem} 

\begin{proof}
Fix a choice of Sylow $p$-subgroups, $\S(v) \in$ Syl$_p(\G(v))$ and $\S(e) \in$ Syl$_p(\G(e))$ for each vertex $v$ and edge $e$ of $\T$. If $v$ is incident to $e$ then by Sylow's Theorem there exists an element $g_{ev} \in \G(v)$ such that $(\S(e)\G(f_{ev}))^{g_{ev}} \leq \S(v).$ Let $\S$ be the functor from $\T$ to $\mathfrak{Grp}$ which sends $e$ and $v$ respectively to $\S(e)$ and $\S(v)$ and $f_{ev} \in$ Hom$_\T(e,v)$ to $\G(f_{ev}) \circ c_{g_{ev}} \in$ Hom$(\S(e),\S(v))$. Then $f_{ev}$ determines a tree of finite $p$-groups $(\T,\S)$. Now let $\F_{\S(-)}(\G(-))$ be the functor from $\T$ to $\mathfrak{Fus}$ which sends $e$ and $v$ respectively to $\F_{\S(e)}(\G(e))$ and $\F_{\S(v)}(\G(v))$ and $f_{ev}$ to the homomorphism $\Phi_{ev}:=\G(f_{ev}) \circ c_{g_{ev}}|_{\S(e)} \in$ Hom$(\S(e),\S(v)).$ Clearly $\Phi_{ev}$ is a morphism of fusion systems and hence determines a tree of fusion systems $(\T,\F_{\S(-)}(\G(-)),\S(-))$, as required.
\end{proof}
The proof of Lemma \ref{induce} shows that there may be many trees of fusion systems $(\T,\F_{\S(-)}(\G(-)),\S(-))$ induced by a tree of groups $(\T,\G)$, since a choice for the functor $\S(-)$ is made when Sylow's Theorem is applied. We need to show that this choice does not interfere with the isomorphism type of the completion $\F_\T$. To do this, we first isolate precisely how conjugation takes place in $\G_\T$ by proving the following lemma of Robinson (\cite[Lemma 1]{Rob}).

\begin{Lem}\label{roblem}
Let $\T$ be a tree consisting of two vertices $v$ and $w$ with a single edge $e$ connecting them.  Let $(\T,\G)$ be a tree of groups with completion $\G_\T$ and write $A:=\G(v)$, $B:=\G(w)$ and $C:=\G(e).$ The following hold:
\begin{itemize}
\item[(a)] Any product of elements whose successive terms lie alternately in the sets $A \backslash C$ and $B \backslash C$, lies outside of $C$, and outside at least one of $A$ and $B$.
\item[(b)] Each element $g \in G \backslash A$ may be written as a product $g=a_0\omega b_{\infty}$ with $a_0 \in A$ and $b_{\infty} \in B$ so that either $\omega=1$ or $$\omega=\prod_{i=1}^s b_ia_i,$$
where $a_i \in A \backslash C$ and $b_i \in B \backslash C$ for $1 \leq i \leq s$. 
\end{itemize}
Consequently, if $X \leq A$, $g \in G \backslash A$ and $X^g \leq A$ or $X^g \leq B$, then writing $g$ as a product $g=a_0b_1a_1 \ldots b_sa_sb_{\infty}$ as in (b), we have $$\langle X_0,X_i,Y_i \mid 1 \leq i \leq s \rangle \leq C$$ where $X_0=X^{a_0},Y_1=X_0^{b_1},X_1=Y_1^{a_1},$ and so on.
\end{Lem}

\begin{proof}
To see (a), let $w=g_0\ldots g_n$ be a product of elements whose successive terms lie alternately in the sets $A \backslash C$ and $B \backslash C$. Then $w$ is a reduced word in $\G_\T$ so if $w$ lies in $C$, it is no longer reduced, (being representable by an element of $C$). Observe that $g_0$ and $g_n$ dictate where $w$ lies. To see this, note that if $g_0,g_n \in A$ then $w \notin B \backslash C$, since otherwise $w=b \in B \backslash C$ implies that $$1=wb^{-1}=g_0,\ldots g_nb^{-1} \notin C,$$ a contradiction. Similarly if $g_0,g_n \in B$ then $w \notin A \backslash C$ and in the remaining cases (where $g_0$ and $g_n$ lie in different sets), $w \notin (A \cup B) \backslash C$. This proves (a).
To see (b), note that certainly any element $g \in G \backslash A$ may be written in the stated way, since (by (a)) such a representation allows for all possibilities for the set in which the element $g$ lies.
Finally, we prove the last assertion of the lemma. If $b_\infty \in C$ and equals $c$ say, then $X^{gc^{-1}}=(X^g)^{c^{-1}} \in A \cup B$ if and only if $X^g \in A \cup B.$ This proves that we may assume without loss of generality that either $b_\infty = 1$ (if $X^g \leq B$) or $b_\infty \in B \backslash C$ (if $X^g \leq A$.) In either case, suppose that there exists $u \in X^{a_0} \backslash C$. Then $$b_\infty^{-1}a_s^{-1} \ldots, b_1^{-1}ub_1\ldots a_sb_{\infty}$$ lies in $A$ or $B$ by assumption which contradicts (a). We obtain a similar contradiction if $X_0^{b_1}$ is not contained in $C$. Inductively, we arrive at the stated result.
\end{proof}
We can now apply Lemma \ref{roblem} to prove Theorem A, relating the two ways in which one can construct a fusion system from a tree of groups.

\begin{Thm}\label{treesgroupthm}
Let $(\T,\G)$ be a tree of finite groups and write $\G_\T$ for the completion of $(\T,\G)$. Let $(\T,\F,\S)$ be a tree of fusion systems induced by $(\T,\G)$ which satisfies $(H)$ so that there exists a completion $\F_\T$ for $(\T,\F,\S)$. The following hold:
\begin{itemize}
\item[(a)] $\S(v_*)$ is a Sylow $p$-subgroup of $\G_\T$.
\item[(b)] $\F_{\S(v_*)}(\G_\T)= \F_\T$.
\end{itemize}
In particular, $\F_\T$ is independent of the choice of tree of fusion systems $(\T,\F,\S)$ induced by $(\T,\G)$.
\end{Thm}

\begin{proof}
We first prove that $S:=S(v_*)$ is a Sylow $p$-subgroup of $G_\T$. Let $P$ be a finite $p$-subgroup of $\G_\T$ and consider the image $\I$ of $\tilde{\T}^P$ in $\T$ under the composite
$$\begin{CD}
\tilde{\T} @>>> \tilde{\T}/{\G_\T} @>\simeq>> \T.\\
\end{CD}$$
Since $\tilde{\T}$ is a tree, $\tilde{\T}^P$ is also a tree, so that $\I$ must be connected. Let $v$ be a vertex in $\I$ which is of minimal distance from $v_*$ and assume that $v \neq v_*$. This implies that there is some $g \in \G_\T$ such that $g\G(v) \in (\G_\T/\G(v))^P$ or equivalently such that $g^{-1}Pg \leq \G(v)$. Let $e$ be the edge $(v,w)$ incident to $v$ in the unique minimal path from $v$ to $v_*$. Since $(\T,\F,\S)$ satisfies $(H)$, $p \nmid |\G(v):\G(e)|$, so by Sylow's Theorem there is $g' \in \G_\T$ such that $g'^{-1}Pg' \leq \G(e)$. Since $\G(e) \leq \G(w)$, we also have $g'\G(w) \in (\G_\T/\G(w))^P$, so that $w$ is a vertex closer to $v_*$ than $v$ in $\I$, a contradiction. Hence $v=v_*$, $g^{-1}Pg \leq \G(v_*)$ and by Sylow's Theorem there is some $g'' \in \G_\T$ with $g''^{-1}Pg'' \leq S$, as needed.

Next we prove that (b) holds. Observe that $\F_\T \subseteq \F_S(\G_\T)$ since (by definition) each morphism in $\F_\T$ is a composite of restrictions of morphisms in $\F(v)$, each of which clearly lies in $\F_S(\G_\T)$. Hence it remains to prove that $\F_S(\G_\T) \subseteq \F_\T$. We proceed by induction on the number of vertices $n:=|V(\T)|$ of $\T$, the result being clear in the case where $n=1$. Suppose that $n > 1$ and fix an extremal vertex, $v$ of $\T$ not equal to $v_*$. Let $\T'$ be the tree obtained from $\T$ by removing $v$ and the unique edge $e$ to which $v$ is incident. Then $(\T',\G)$ is a graph of groups and $(\T',\F,\S)$ is a tree of fusion systems induced by $(\T',\G)$ which satisfies $(H)$. Furthermore, $S \in$ Syl$_p(\G_{\T'})$ and by induction, the completion $\F_{\T'}$ of $(\T',\F,\S)$ is the fusion system $\F_S(\G_{\T'})$. Since $$\G_{\T}=\G_{\T'} *_{\G(e)} \G(v) \mbox{ and } \F_\T=\langle \F_{\T'}, \F_{\S(v)}(\G(v)) \rangle,$$ it will be enough to show that $$\F_S(\G_{\T'} *_{\G(e)} \G(v)) \subseteq \langle \F_{\T'}, \F_{\S(v)}(\G(v)) \rangle. $$ 

Let $X \leq S$ and suppose that $\langle X,X^g \rangle \leq S$ with $g \in \G_\T.$ If $g \in \G_{\T'}$ then we are done, so suppose that $g \notin \G_{\T'}$ and write $$g=g_{-\infty}h_1g_1 \ldots h_rg_rh_{r+1}g_{\infty}$$ where $g_{-\infty},g_{\infty} \in \G_{\T'}$, $g_i \in \G_{\T'} \backslash \G(e)$ and $h_i \in \G(v) \backslash \G(e).$ Set $X^*:=X^{g_{-\infty}}.$ Then $\langle X^*,(X^*)^{h_1g_1 \ldots h_rg_rh_{r+1}} \rangle \leq \G_{\T'}$ so that by Lemma $\ref{roblem}$, $$X^*, (X^*)^{h_1}, (X^*)^{h_1g_1},\ldots,(X^*)^{h_1g_1\ldots h_rg_rh_{r+1}} \leq \G(e).$$ 

By Sylow's Theorem there exists $k_0 \in \G(e)$ such that $(X^*)^{k_0} \leq \S(e)$, $l_1 \in \G(e)$ such that $(X^*)^{h_1l_1} \leq \S(e)$, $k_1 \in \G(e)$ such that $(X^*)^{h_1g_1k_1} \leq \S(e)$, and so on. Furthermore the sequence of elements $$k_0, k_0^{-1}h_1l_1, l_1^{-1}g_1k_1,k_1^{-1}h_2l_2, \ldots, k_r^{-1}h_{r+1}l_{r+1},l_{r+1}^{-1}g_\infty$$ lie alternately in the groups $\G_{\T'}$ and $\G(v)$ and multiply to give $h_1g_1\ldots h_rg_rh_{r+1}g_{\infty}$. Now, for $1 \leq i \leq r,$ conjugation from $(X^*)^{h_1...h_il_i}$ to $(X^*)^{h_1...h_ig_ik_i}$ is carried out in the fusion system of $\G_{\T'}$ and conjugation from $(X^*)^{h_1...h_ig_ik_i}$ to $(X^*)^{h_1...g_ih_{i+1}l_{i+1}}$ occurs in the fusion system of $\G(v).$ Lastly, conjugation from $(X^*)^{h_1...h_rg_rh_{r+1}l_{r+1}}$ to $(X^*)^{h_1...h_rg_rh_{r+1}g_{\infty}}$ is carried out in the fusion system of $\G_{\T'}$. Since $\S(e)=\S(v)$, the result follows.
\end{proof}

We end this section with a short example which illustrates how Theorem \ref{treesgroupthm} can be applied to produce fusion systems from group amalgams.

\begin{Ex}\label{ftex}
Let $H \cong$ Sym$(4)$, $S \cong D_8$ be a Sylow 2-subgroup of $H$ and let $V_1$ and $V_2$ be the two non-cyclic non-conjugate subgroups of order 4 in $S$. Also, let $\T$ be a tree consisting of two vertices $v_1$ and $v_2$ with a single edge $e$ connecting them. Form a tree of groups $(\T,\G)$ where $\G(v_1) =\G(v_2) = H$ and $\G(e)= S$ and where $V_i f_{ev}$ is normal in $\G(v_i)$ but not normal in $\G(v_{3-i})$ for $i=1,2.$ Let $(\T,\F,\S)$ be a tree of fusion systems induced by $(\T,\G)$ and observe that $(\T,\F,\S)$ satisfies $(H)$ so that it has a completion $\F_\T$ by Lemma \ref{compft}. It is well known that the group $PSL_3(2)$ is a faithful, finite completion of $(\T,\G)$ (see, for example, \cite[Theorem A]{Gol}) so that by Theorem \ref{treesgroupthm} we must have that $\F_\T$ is the fusion system of $PSL_3(2)$. 
\end{Ex}

\subsection{The $P$-orbit Graph}
We next consider what can said about the $P$-orbit graph in the special case where a tree of fusion systems $(\T,\F,\S)$ which satisfies $(H)$ is induced by a tree of groups $(\T,\G).$ It is shown that like the completion $\F_\T$, the isomorphism type of Rep$_{\F_\T}(P,\F)$ is independent of the choice of $(\T,\F,\S)$ since it may described in terms of the orbit graph $\tilde{\T}$ of $(\T,\G)$. 

We start by introducing some more notation. If $A$ and $B$ are groups, let Rep$(A,B)$ denote the set of orbits of the action of Inn$(B)$ on Hom$(A,B)$ by right composition. For each $\alpha \in$ Hom$(A,B)$, let $[\alpha]_B$ denote its class modulo Inn$(B)$. If $A$ and $B$ are subgroups of some group $C$, write Rep$_C(A,B)$ for the set of orbits of the action of Inn$(B)$ on Hom$_C(A,B) \cong N_C(A,B)/C_C(A)$ (where $N_C(A,B)=\{g \in C \mid A^g \leq B\}$ and $C_C(A)$ acts on $N_C(A,B)$ via left multiplication).

Now let $(\T,\G)$ be a tree of groups and fix a vertex $v_*$ of $\T$. For each $P \leq \G(v_*)$ let Rep$(P,\G(-))$ be the functor from $\T$ to $\mathfrak{Set}$ which sends vertices $v$ and edges $e$ respectively to Rep$(P,\G(e))$ and Rep$(P,\G(e))$ and which sends $f_{ev}$ to the map from Rep$(P,\G(e))$ to Rep$(P,\G(v))$ given by $$[\alpha]_{\G(e)} \longmapsto [\alpha \circ \iota_{\G(e)}^{\G(v)}]_{\G(v)}.$$ Similarly if $\G_\T$ is the completion of $(\T,\G)$ let Rep$_{\G_\T}(P,\G(-))$ be the functor from $\T$ to $\mathfrak{Set}$ which sends vertices $v$ and edges $e$ respectively to Rep$_{\G_\T}(P,\G(v))$ and Rep$_{\G_\T}(P,\G(e))$ and which sends $f_{ev}$ to the map given by $$[\alpha]_{\G(e)} \longmapsto [\alpha \circ \iota_{\G(e)}^{\G(v)}]_{\G(v)}.$$ (Observe that this well-defined). Now define: $$\mbox{Rep}(P,\G):= \underrightarrow{\mbox{hocolim}}_{
  \substack{
  \T}} \Rep(P,\G(-)) \mbox{, } \mbox{Rep}_{\G_\T}(P,\G):= \underrightarrow{\mbox{hocolim}}_{
  \substack{
  \T}} \Rep_{\G_\T}(P,\G(-))$$ both regarded as graphs in the usual way.
The next proposition allows us to compare Rep$_{\G_\T}(P,\G)$ with $\tilde{\T}$.
\begin{Prop}\label{phip}
Let $(\T,\G)$ be a tree of finite groups. Fix a vertex $v_*$ of $\T$ and a subgroup $P \leq \G(v_*)$ and write $\G_\T$ for the completion of $(\T,\G)$. There exists a graph isomorphism $$\begin{CD}
\tilde{\T}^P/C_{\G_\T}(P) @>\sim>> \Rep_{\G_\T}(P,\G),\\
\end{CD}$$ where $\tilde{\T}^P$ is the subgraph of $\tilde{\T}$ fixed under the action of $P$.
\end{Prop}

\begin{proof}
Define a map $$\begin{CD}
\Phi_P:\tilde{\T}^P @>>> \mbox{Rep}_{\G_\T}(P,\G)\\
\end{CD}$$ by sending a vertex $g\G(v)$ to $[c_g]_{\G(v)}$ where $c_g \in$ Hom$_{\G_\T}(P,\G(v))$ and similarly by sending an edge $g\G(e)$ to $[c_g]_{\G(e)}$ where $c_g \in$ Hom$_{\G_\T}(P,\G(e)).$ Note that this makes sense: if $g\G(v)$ is a vertex in $\tilde{\T}^P$ then for each $x \in P$, $xg\G(v)=g\G(v)$ which is equivalent to $P^g \leq \G(v)$. The same argument works for edges. To see that this map defines a homomorphism of graphs, note that if $g\G(v)$ and $h\G(w)$ are connected via an edge $i\G(e)$ in $\tilde{\T}^P$ then $g\G(v)=i\G(v)$ and $h\G(w) \in i\G(w)$ so that [$c_g]_{\G(v)}=[c_i \circ \iota_{\G(e)}^{\G(v)}]_{\G(v)}$ and $[c_h]_{\G(w)}=[c_i \circ \iota_{\G(e)}^{\G(v)}]_{\G(w)}.$ It is clear from the definition that $\Phi_P$ is surjective. If $[c_g]_{\G(v)}=[c_h]_{\G(w)}$ then $v=w$ and there is some $r \in \G(v)$ such that $c_g=c_h \circ c_r \in$ Hom$_{\G_\T}(P, \G(v))$. This is equivalent to the orbits of $g$ and $hr$ under the action of $C_{\G_\T}(P)$ on $N_{\G_\T}(P,\G(v))$ being equivalent so that there is some $x \in C_{\G_\T}(P)$ with $xg=hr$. Hence $h^{-1}xg \in \G(v)$, $h\G(v)=xg\G(v)$ and the vertices $h\G(v)$ and $g\G(v)$ lie in the same $C_{\G_\T}(P)$-orbit, as needed.
\end{proof}

We can now compare the graphs Rep$(P,\G)$ and Rep$(P,\F)$ when $\F$ is a choice of functor $\F_{\S(-)}(\G(-))$ associated to a tree of finite groups $(\T,\G)$ (see Lemma \ref{induce}). We need the following simple consequence of Sylow's Theorem.

\begin{Lem}\label{ftgt}
Let $G$ be a finite group and let $S$ be a Sylow $p$-subgroup of $G$. Then for each $p$-group $P$, there exists a bijection $$\begin{CD}
\Rep(P,\F_S(G)) @>\sim>> \Rep(P,G)\\
\end{CD}$$ given by the map $\Psi$ which sends $[\alpha]_{\F_S(G)}$ to $[\alpha \circ \iota_S^G]_G$ for each $\alpha \in \Hom(P,S).$
\end{Lem}

\begin{proof}
Let $\alpha,\beta \in$ Hom$(P,S)$ and suppose that $[\alpha \circ \iota_S^G]_G=[\beta \circ \iota_S^G]_G$. Then there exists a map $c_g \in$ Inn$(G)$ such that $\alpha \circ \iota_S^G \circ c_g=\beta \circ \iota_S^G$. This implies that $c_g|_S \in$ Hom$_{\F_S(G)}(P\alpha,P\beta)$ and $\alpha \circ c_g|_S = \beta$ so that $[\alpha]_{\F_S(G)}=[\beta]_{\F_S(G)}$ and $\Psi$ is injective. To see that $\Psi$ is surjective, let $[\gamma]_G  \in$ Rep$(P,G)$. By Sylow's Theorem there exists $g \in G$ such that $(P\gamma)^g \leq S$. Set $\varphi:=\gamma \circ c_g \in$ Hom$(P,S)$ and observe that $[\varphi \circ \iota_S^G]_G=[\gamma]_G,$ as required. 
\end{proof}

\begin{Cor}\label{corftgt}
Let $(\T,\G)$ be a tree of finite groups with completion $\G_\T$ and let $(\T,\F,\S)$ be a tree of fusion systems induced by $(\T,\G)$. Fix a vertex $v_*$ of $\T$ and a subgroup $P \leq \S(v_*)$. Then there exists a natural isomorphism of functors $$\begin{CD}
\Rep(P,\F(-)) @>\sim>> \Rep(P,\G(-))\\
\end{CD},$$ which induces a homotopy equivalence $\begin{CD}
\Rep(P,\F) @>\simeq>> \Rep(P,\G)\\
\end{CD}.$ Furthermore, if $(\T,\F,\S)$ satisfies $(H)$ and has completion $\F_\T$ then $$\begin{CD}
\Rep_{\F_\T}(P,\F) @>\simeq>> \Rep_{\G_\T}(P,\G)\\
\end{CD},$$ so that $\Rep_{\F_\T}(P,\F)$ is independent of the choice of tree of fusion systems $(\T,\F,\S)$.
\end{Cor}

\begin{proof}
The existence of a natural isomorphism of functors $$\begin{CD}
\eta:\mbox{Rep}(P,\F(-)) @>\sim>> \mbox{Rep}(P,\G(-))\\
\end{CD}$$ is an immediate consequence of Lemma \ref{ftgt}. By \cite[Proposition IV.1.9]{GJ} this natural isomorphism induces a homotopy equivalence $$\begin{CD}
\mbox{hocolim}(\eta): \mbox{Rep}(P,\F) @>\simeq>> \mbox{Rep}(P,\G)\\
\end{CD}.$$  If $(\T,\F,\S)$ satisfies $(H)$ then the above equivalence sends the vertex $[\iota_P^{\S(v_*)}]_{\F(v_*)}$ to $[\iota_P^{\G(v_*)}]_{\G(v_*)}.$ Clearly $[\iota_P^{\G(v_*)}]_{\G(v_*)}$ lies in Rep$_{\G_\T}(P,\G)$ (it is the image of the coset $\G(v_*)$ under $\Phi_P$). Since Rep$_{\G_\T}(P,\G)$ is connected (being the image under $\Phi_P$ of $\tilde{\T}^P$ which is connected), hocolim($\eta$) must restrict to a homotopy equivalence $$\begin{CD}
\mbox{Rep}_{\F_\T}(P,\F) @>\simeq>> \mbox{Rep}_{\G_\T}(P,\G)\\
\end{CD}$$ by Proposition \ref{phipft}. This completes the proof.
\end{proof}

\section{The Completion of a Tree of Fusion Systems}\label{compfus}
In this section, we prove our second main result, Theorem B, which gives conditions for a tree of fusion systems $(\T,\F,\S)$ to have a saturated completion $\F_\T$:

\begin{Thm}\label{completionsat}
Let $(\T,\F,\S)$ be a tree of fusion systems which satisfies $(H)$ and assume that $\F(v)$ is saturated for each vertex $v$ of $\T.$ Write $S:=\S(v_*)$ and $\F_\T$ for the completion of $(\T,\F,\S)$. Assume that the following hold for each $P \leq S$.
\begin{itemize}
\item[(a)] If $P$ is $\F_\T$-conjugate to an $\F(v)$-essential subgroup or $P=\S(v)$ then $P$ is $\F_\T$-centric.
\item[(b)] If $P$ is $\F_\T$-centric then $\Rep_{\F_\T}(P,\F)$ is a tree. 
\end{itemize}
Then $\F_\T$ is a saturated fusion system on $S$.
\end{Thm}
Conditions (a) and (b) are respectively motivated by Theorems \ref{alpthm} and \ref{alpthmconv}. The former condition implies that $\F_\T$ is generated by its $\F_\T$-centric subgroups, while the latter ensures that the saturation axioms (Definition \ref{sat} (a) and (b)) hold for all such subgroups. Thus the saturation of $\F_\T$ will ultimately follow from Theorem \ref{alpthmconv}. Theorem \ref{completionsat} was inspired by a theorem of Broto, Levi and Oliver (\cite[Theorem 4.2]{BLO4}) and we will deduce their result from ours in Corollary \ref{completionsatgroups}.

\subsection{Finite Group Actions on Trees}
Before embarking on the proof of Theorem \ref{completionsat}, it will be important to understand the way in which the automiser Aut$_{\F_\T}(P)$ can act on the $P$-orbit graph Rep$_{\F_\T}(P,\F)$. Note that we may draw an analogy here with the situation for trees of groups $(\T,\G)$ where there exists an action of the completion $\G_\T$ on the orbit tree $\tilde{\T}$. We begin with two important concepts which play a pivotal role in the proof of Theorem \ref{completionsat}.

\begin{Def}
Let $(\T,\F,\S)$ be a tree of fusion systems which satisfies $(H)$ and write $\F_\T$ for the completion of $(\T,\F,\S)$. 

\begin{itemize}
\item[(a)] For each pair of $p$-groups $P \leq Q$ define the \textit{restriction map} from $Q$ to $P$: 
$$\begin{CD}
\mbox{res}^Q_P:\mbox{Rep}_{\F_\T}(Q,\F) @>>> \mbox{Rep}_{\F_\T}(P,\F)\\
\end{CD}$$
to be the map which sends a vertex $[\varphi]_{\F(v)}$ in Rep$_{\F_\T}(Q,\F)$ to $[\varphi|_P]_{\F(v)}$ in Rep$_{\F_\T}(P,\F)$. 
\item[(b)] For each $P \leq \S(v_*)$ the \textit{action} of $\F_\T$ on Rep$_{\F_\T}(P,\F)$ is given by the group action of Aut$_{\F_\T}(P)$ on Rep$_{\F_\T}(P,\F)$ which sends a vertex $[\varphi]_{\F(v)}$ to $[\psi \circ \varphi]_{\F(v)}$ for each $\psi \in$ Aut$_{\F_\T}(P)$.
\end{itemize} 
\end{Def}
We quickly check that this definition make sense. Firstly, one observes that res$^Q_P$ defines a homomorphism of graphs. To see this, note that if $[\alpha]_{\F(v)}$ is incident to some edge $[\beta]_{\F(e)}$ in Rep$_{\F_\T}(Q,\F)$, then there exists $\psi \in$ Hom$_{\F(v)}(Q\alpha,\S(v))$ such that $\beta \circ \iota_{\S(e)}^{\S(v)}=\alpha \circ \psi$. Hence $\beta|_P \circ \iota_{\S(e)}^{\S(v)}=\alpha|_P \circ \psi|_{P\alpha}$ and $[\alpha|_P]_{\F(v)}$ is incident to $[\beta|_P]_{\F(e)}$ in Rep$_{\F_\T}(P,\F)$. Secondly, note the described action of  Aut$_{\F_\T}(P)$ on Rep$_{\F_\T}(P,\F)$ makes sense: if $[\varphi]_{\F(v)}$ is a vertex of Rep$_{\F_\T}(P,\F)$ and $\psi \in$ Aut$_{\F_\T}(P)$, then $\psi \circ \varphi \in$ Hom$_{\F_\T}(P,\S(v))$ so that $[\psi \circ \varphi]_{\F(v)}$ is also a vertex of Rep$_{\F_\T}(P,\F)$.

The following proposition gives conditions on $(\T,\F,\S)$ which will allow us to calculate the image of res$^Q_P$ in terms of certain fixed points of the action of $\F_\T$ on Rep$_{\F_\T}(P,\F)$.

\begin{Prop}\label{resprop}
Let $(\T,\F,\S)$ be a tree of fusion systems which satisfies $(H)$ and assume that the following hold:
\begin{itemize}
\item[(a)] $\F(v)$ is saturated for each vertex $v$ of $\T;$ and
\item[(b)] $\F(e)$ is a trivial fusion system $\F_{\S(e)}(\S(e))$ for each edge $e$ of $\T$.
\end{itemize}
Let $P$ and $Q$ be $p$-groups with $P$ $\unlhd$ $Q \leq S(v_*)$ and assume that $P$ is $\F_\T$-centric, where $\F_\T$ is the completion of $(\T,\F,\S)$. If $\Rep_{\F_\T}(P,\F)$ is a tree, then the image of the restriction homomorphism 
$$\begin{CD}
\res^Q_P:\Rep_{\F_\T}(Q,\F) @>>> \Rep_{\F_\T}(P,\F)\\
\end{CD}$$ is equal to $\Rep_{\F_\T}(P,\F)^{\Aut_Q(P)}$. 
\end{Prop}

\begin{proof}
Set $K:=$ Aut$_Q(P)$. If $[\alpha]$ is a vertex or edge in Rep$_{\F_\T}(Q, \F)$ then $[\alpha|_P] \in$ Rep$_{\F_\T}(P, \F)^K$ since $$c_g \circ \alpha|_P = \alpha|_P \circ (\alpha^{-1}|_{P\alpha} \circ c_g \circ \alpha|_P)$$ for each $g \in Q$. It remains to prove that each vertex or edge of Rep$_{\F_\T}(P, \F)^K$ is the image of \textit{some} vertex or edge of Rep$_{\F_\T}(Q, \F)$. Since Rep$_{\F_\T}(P, \F)$ is a tree and $K$ is finite, Rep$_{\F_\T}(P, \F)^K$ is also a tree. In particular Rep$_{\F_\T}(P, \F)^K$ is connected. Hence to complete the proof, it will be enough to show that for each edge $[\beta]_{\F(e)}$ in Rep$_{\F_\T}(P, \F)^K$ which is connected to the image $[\alpha|_P]_{\F(v)}$ of some vertex $[\alpha]_{\F(v)}$ under res$_P^Q$, there exists an edge $[\bar{\beta}]_{\F(e)}$ in Rep$_{\F_\T}(Q,\F)$ connected to $[\alpha]_{\F(v)}$ whose image under res$_P^Q$ is $[\beta]_{\F(e)}$. Since $[\iota_Q^S]_{\F(v_*)}$ gets sent to $[\iota_P^S]_{\F(v_*)}$ under res$_P^Q$, we may assume (inductively) that we have chosen $e$ to be the edge to which $v$ is incident in the unique minimal path from $v$ to $v_*$. In particular we may assume that $\S(e)=\S(v)$.  By assumption, there exists $\psi \in$ Hom$_{\F(v)}(P\beta,P\alpha)$ such that $\beta \circ \psi=\alpha|_P$. Set $R:=N_{\S(e)}^{K^\beta}(P\beta)$. Since $[\beta]_{\F(e)}$ is fixed by $K$, $K^\beta \leq$ Aut$_{\S(e)}(P\beta)$ so Aut$_R(P\beta)=K^{\beta}$. Now, Aut$_R(P\beta)^\psi=K^{\beta\psi}=K^\alpha=$ Aut$_{Q\alpha}(P\alpha)$ and hence $$Q\alpha \leq N_{\psi^{-1}}=\{g \in N_{\S(v)}(P\alpha) \mid (c_g)^{\psi^{-1}} \in \mbox{ Aut}_S(P\beta)\}.$$ Since $\F(v)$ is saturated and $P\alpha$ is $\F(v)$-centric (and therefore fully $\F(v)$-centralised), $\psi^{-1}$ extends to a map $\rho \in $ Hom$_{\F(v)}(Q\alpha, \S(v))$. Set $\bar{\beta}:=\alpha \circ \rho \in$ Hom$_{\F(v)}(Q,\S(v))$ so that res$_P^Q([\bar{\beta}]_{\F(e)})= [\bar{\beta}|_{P}]_{\F(e)}=[\beta]_{\F(e)}$ and $[\alpha]_{\F(v)}$ is incident to $[\bar{\beta}]_{\F(e)}$ in Rep$_{\F_\T}(Q,\F)$. This completes the proof.
\end{proof}

\subsection{Proof of Theorem B}
We now have all of the tools necessary to prove Theorem B. We remind the reader that this theorem was originally conjectured based on the corresponding statement for trees of group fusion systems, \cite[Theorem 4.2]{BLO4}, and the proof in that case may be seen to rely on some fairly deep homotopy theory. No deep results are required in our proof and for this reason we feel the statement is more naturally understood in the context of trees of fusion systems. Nevertheless, we will present \cite[Theorem 4.2]{BLO4} as Corollary \ref{completionsatgroups} below.

A key step in our proof is the observation that the completion $\F_\T$ of a tree of fusion systems $(\T,\F,\S)$ is independent of where $\F$ sends edges $e$ of $\T$, provided $\F_{\S(e)}(\S(e)) \subseteq \F(e) \subseteq \F(v)$ (inside $\F_\T$). This means that we are able to assume that $\F(e)$ is the trivial fusion system $\F_{\S(e)}(\S(e))$. 

\begin{proof}[Proof of Theorem B]
Let $(\T,\F',\S)$ be the tree of fusion systems obtained from $(\T,\F,\S)$ by setting $\F'(v):=\F(v)$ for each vertex $v$ of $\T$ and $\F'(e):=\F_{\S(e)}(\S(e))$ for each edge $e$ of $\T.$ Then $(\T,\F',\S)$ satisfies $(H)$ and has completion $\F_\T$. Also, since any cycle in Rep$_{\F_\T}(P,\F')$ is a cycle in Rep$_{\F_\T}(P,\F)$, hypothesis (b) in the theorem remains true for $(\T,\F',\S).$ Hence we may assume from now on that $(\T,\F,\S)=(\T,\F',\S)$.

Since $\F(v)$ is a saturated fusion system on $\S(v)$ for each vertex $v$ in $\T$, by Theorem \ref{alpthm} $$\F(v)=\langle  \{\mbox{Aut}
_{\F(v)}(P) \mid P \mbox{ is $\F(v)$-essential }\}\cup \{\mbox{Aut}_{\F(v)}(\S(v)) \} \rangle_{\S(v)}.$$ 
By hypothesis, each $\F(v)$-essential subgroup and each $\S(v)$ is an $\F_\T$-centric subgroup so by the definition of $\F_\T$ we have  $$\F_\T=\langle  \mbox{Aut}
_{\F(v)}(P) \mid v \in V(\T), P \mbox{ is $\F_\T$-centric} \rangle_S.$$
Hence by Theorem \ref{alpthmconv} it suffices to verify that axioms (a) and (b) in Definition \ref{fus} hold for $\F_\T$-centric subgroups.

First we prove Definition \ref{fus} (a) holds. Let $P$ be a fully $\F_\T$-normalised, $\F_\T$-centric subgroup of $S$. Then since $|C_S(P\varphi)|=|Z(P\varphi)|=|Z(P)|$ for each $\varphi \in$ Hom$_{\F_\T}(P,S)$, $P$ is certainly fully $\F_\T$-centralised.  It remains to prove that $P$ is fully $\F_\T$-automised. Each $\varphi \in$ Aut$_{\F_\T}(P)$ acts on Rep$_{\F_\T}(P,\F)$ by sending a vertex $[\alpha]_{\F(v)}$ to $[\varphi \circ \alpha]_{\F(v)}$. Since Aut$_{\F_\T}(P)$ is finite and Rep$_{\F_\T}(P,\F)$ is a tree, there exists a vertex or edge $[\alpha]_{\F(r)}$ in Rep$_{\F_\T}(P,\F)$ which is fixed under the action of Aut$_{\F_\T}(P)$ ($r \in V(\T) \cup E(\T)$). This means that Aut$_{\F(r)}(P\alpha)=$ Aut$_{\F_\T}(P)^\alpha$. 

Choose $\beta \in$ Hom$_{\F(r)}(P\alpha, S(r))$ so that $P\alpha\beta$ is fully $\F(r)$-normalised. Since $\F(r)$ is saturated, Aut$_{S(r)}(P\alpha\beta) \in$ Syl$_p($Aut$_{\F(r)}(P\alpha\beta))$. Now $$|\mbox{Aut}_S(P)|=|N_S(P)|/|Z(P)| \geq |N_S(P\alpha\beta)|/|Z(P\alpha\beta)| \geq |\mbox{Aut}_{S(r)}(P\alpha\beta)|,$$ where the first inequality follows from the fact that $P$ is fully $\F_\T$-normalised. Since $$|\mbox{Aut}_{\F(r)}(P\alpha\beta)|=|\mbox{Aut}_{\F(r)}(P\alpha)|=|\mbox{Aut}_{\F_\T}(P)|,$$ we must have Aut$_S(P) \in$ Syl$_p($Aut$_{\F_\T}(P))$, so that $P$ is fully $\F_\T$-automised, as needed.

Next we prove that (b) holds in Definition \ref{fus}.  Fix $\varphi \in$ Hom$_{\F_\T}(P,S)$ where $P$ is $\F_\T$-centric. We need to show that there is some $\bar{\varphi} \in$ Hom$_{\F_\T}(N_\varphi, S)$ which extends $\varphi$. Set $K:=$ Aut$_{N_\varphi}(P)=$ Aut$_S(P)$ $\cap$ Aut$_S(P\varphi)^{\varphi^{-1}}.$ If $c_g \in K$ then $c_g \circ \varphi=\varphi \circ c_h$ for some $h \in$ Aut$_S(P\varphi)$ and this show that the vertex $[\varphi]_{\F(v_*)}$ in Rep$_{\F_\T}(P,\F)$ is fixed by the action of $K$. By Proposition \ref{resprop} applied with $Q=N_\varphi$ there is some $[\psi]_{\F(v_*)} \in$ Rep$_{\F_\T}(N_\varphi, \F)$ with $[\psi|_P]_{\F(v_*)}=[\varphi]_{\F(v_*)}$. This implies that there is $\rho \in$ Iso$_{\F(v_*)}(P\varphi,P\psi)$ such that $\psi|_P=\varphi \circ \rho$. Since $\F(v_*)$ is saturated, $\rho^{-1}|_{P\psi}= (\psi|_P)^{-1} \circ \varphi$ extends to a map $\chi \in $ Hom$_{\F(v_*)}((N_\varphi)\psi, S)$. To see this, observe that $(N_\varphi)\psi \leq N_{\rho^{-1}}$ since for $g \in N_\varphi$, $(c_{g\psi})^{\rho^{-1}}=(c_g)^\varphi \in$ Aut$_S(P\varphi)=$ Aut$_S(P\psi\rho^{-1}).$ Hence we set $\bar{\varphi}:= \psi \circ \chi \in$ Hom$_\F(N_\varphi,S)$ so that $\bar{\varphi}$ extends $\varphi$, as needed. This completes the proof of (b) in Definition \ref{fus} and hence the proof of the theorem.
\end{proof}

\begin{Cor}\label{completionsatgroups}
Let $(\T,\G)$ be a tree of groups with completion $\G_\T$ and assume that there exists a vertex $v_*$ of $\T$ for which the following is true: for each $v \in V(\T)$ not equal to $v_*$, $p \nmid |\G(v):\G(e)|$ where $e$ is the to which $v$ is incident in the unique minimal path from $v$ to $v_*$. Write $S:=\S(v_*)$ and assume that for each $P \leq S$,
\begin{itemize}
\item[(a)] if $P$ is $\G_\T$-conjugate to an $\F_{\S(v)}(\G(v))$-essential subgroup then $P$ is $\F_S(G_\T)$-centric; and
\item[(b)]  if $P$ is $\F_S(G_\T)$-centric then $\tilde{\T}^P/C_{\G_\T}(P)$ is a tree. 
\end{itemize}
Then $\F_S(\G_\T)$ is saturated.
\end{Cor}

\begin{proof}
Let $(\T,\F,\S)$ be a tree of fusion systems induced by $(\T,\G)$. The hypothesis of the corollary implies that $(\T,\F,\S)$ satisfies $(H)$, so there exists a completion $\F_\T$ of $(\T,\F,\S)$ by Lemma \ref{compft}. By Theorem \ref{treesgroupthm}, $\F_\T \cong \F_S(\G_\T)$, so it suffices to verify conditions (a) and (b) in Theorem \ref{completionsat} hold. Clearly (a) holds by assumption and (b) holds by Proposition \ref{phip} and Corollary \ref{corftgt}. 
\end{proof}

\begin{Ex}
Let $(\T,\F,\S)$ be the tree of fusion systems constructed in Example \ref{ftex}. Clearly (a) holds in the statement of Theorem \ref{completionsat} since $S, V_1$ and $V_2$ are the only $\F_\T$-centric subgroups. It remains to check that (b) holds for these subgroups. Firstly, we show that $|$Rep$_{\F_\T}(V_i,\F(e))|=3$. To see this, for $i=1,2$ write $Z(S)=\langle z \rangle$, choose $x_i$ such that $V_i=\langle z, x_i \rangle$ and notice that $V_i\varphi=V_i$ for each $\varphi \in$ Hom$_{\F_\T}(V_i,S)$. Since Aut$_{\F(e)}(V_i)$ interchanges $x_i$ and $x_iz$, the class of $\varphi$ is determined by which of $x_i,z$ and $x_iz$ gets mapped to $z$ under $\varphi$ yielding exactly $3$ classes of embeddings. By similar arguments the facts that Aut$_{\F(v_i)}(V_i) \cong$ Sym$(3)$ and Aut$_{\F(v_{3-i})}(V_i) \cong C_2$ imply that $|$Rep$_{\F_\T}(V_i,\F(v_i))|=1$ and $|$Rep$_{\F_\T}(V_i,\F(v_{3-i}))|=3$ respectively for $i=1,2$.
Since $4=3+1$, the Euler characteristic of Rep$_{\F_\T}(V_i,\F)$ is 0 for $i=1,2$ so this graph is a tree. In fact, we can see directly that Rep$_{\F_\T}(V_i,\F)$ is a 3 pointed star. Finally, the graph Rep$_{\F_\T}(S,\F)$ obviously consists of a single vertex and so the saturation of $\F_\T$ follows from Theorem \ref{completionsat}.
\end{Ex}

\section{Attaching $p'$-automorphisms to Fusion Systems}\label{extconst}
In this section, we apply Theorem B to prove Theorem C which is a general procedure for extending a fusion system by $p'$-automorphisms of its centric subgroups. Theorem C is a generalisation of \cite[Proposition 5.1]{BLO4} to arbitrary saturated fusion systems.

\subsection{Extensions of $p$-constrained Groups}
We appeal to some cohomological machinery which will allow us to prove a lemma which gives conditions for a $p$-constrained group to be extended by a $p'$-group of automorphisms of a normal centric subgroup. First recall the following characterisation of extensions of non-abelian groups.

\begin{Thm}\label{nonabext}
Let $G$ and $N$ be finite groups. Each homomorphism $$\begin{CD}
\psi_E: G @>>> \Out(N) \end{CD}$$ determines an obstruction in $H^3(G,Z(N))$ which vanishes if and only if there exists an extension
$$\begin{CD}
1 @>>> N @>>> E @>>> G @>>> 1 
\end{CD}$$
which gives rise to $\psi_E$ via the conjugation action of $E$ on $N$.
Furthermore, the group $H^2(G, Z(N))$ acts freely and transitively on the set of (suitably defined) equivalence classes of such extensions.
\end{Thm} 

\begin{proof}
See \cite[Theorems IV.6.6 and IV.6.7]{Br}.
\end{proof}
When $H \leq G$, there is a \textit{restriction map}
$$\begin{CD}
H^n(G,M) @>\mbox{res}_H^G>> H^n(H,M) 
\end{CD}$$ for each $G$-module $M$ and $n \geq 0$ (see \cite[Section III.9]{Br}). Each $g \in G$ induces a well-defined map $\begin{CD}
H^n(H,M) @>g>> H^n(H^g,M) 
\end{CD}$ and we define $z \in H^*(H,M)$ to be \textit{$G$-invariant} if $$\mbox{res}_{H \cap H^g}^H z = \mbox{res}_{H \cap H^g}^{H^g} zg $$ for each $g \in G$. The following result characterises $G$-invariant elements when $M$ is an $\mathbb{F}_pG$-module for some prime $p$.

\begin{Thm}
Let $G$ be a finite group, $p$ be prime and $M$ be an $\mathbb{F}_pG$-module. For each $H \leq G$ with $p \nmid |G:H|$ and $n \geq 0$, res$_H^G$ maps $H^n(G,M)$ isomorphically onto the set of  $G$-invariant elements of $H^n(H,M)$. 
\end{Thm}

\begin{proof}
This follows from \cite[Theorem III.10.3]{Br}.
\end{proof}
As an immediate consequence we have:

\begin{Cor}\label{cohcor}
Let $G$ be a finite group and $M$ be an $\mathbb{F}_pG$-module. If $H$ is a strongly $p$-embedded subgroup of $G$ then res$_H^G$ induces an isomorphism
$$\begin{CD}
H^*(G,M) @>\mbox{res}_H^G>> H^*(H,M). 
\end{CD}$$  
\end{Cor}

Theorem \ref{nonabext} and Corollary \ref{cohcor} may be applied to prove the following result which is a key ingredient in the proof of Theorem C.

\begin{Lem}\label{constex}
Let $H$ be a finite group with $O_{p'}(H)=1$ and let $Q$ be a normal $p$-subgroup of $H$ with $C_H(Q) \leq Q$. Assume $\Delta \leq \Out(Q)$ is chosen such that $\Out_H(Q)$ is a strongly $p$-embedded subgroup of $\Delta$. There exists a finite group $G$ containing $H$ as a $p'$-index subgroup with $Q \unlhd G$ and $\Out_G(Q) \cong \Delta$.
\end{Lem}

\begin{proof}
Write $K:=$ Out$_H(Q)$. Since $K$ is strongly $p$-embedded in $\Delta$ and $Z(Q)$ may be regarded as an $\mathbb{F}_p \Delta$-module, Corollary \ref{cohcor} implies that $H^j(\Delta; Z(Q)) \cong H^j(K;Z(Q))$ for each $j > 0$. Since this holds when $j=3$,  there exists a finite group $G$ which fits into a diagram

$$\begin{CD}
1 @>>> Q @>>> G @>>> \Delta @>>> 1\\
@. @| @AAA @AAA  \\
1 @>>> Q @>>> H @>>> K @>>> 1,
\end{CD}$$
by Theorem \ref{nonabext}.  Since $H^2(\Delta; Z(Q)) \cong H^2(K;Z(Q))$ acts freely and transitively on the set of all such extensions of $Q$ by $\Delta$, we may choose $G$ such that $H \leq G$. In particular, $\Delta \cong G/Q=$ Out$_G(Q)$, as required.
\end{proof}


\subsection{Proof of Theorem C}\label{proofsat}
Let $\F_0$ be a saturated fusion system on a finite $p$-group $S$. Our goal is to apply Lemma \ref{constex} to find conditions for there to exist a saturated fusion system $\F$ containing $\F_0$ with the property that Aut$_{\F_0}(Q)$ is a strongly $p$-embedded subgroup of Aut$_\F(Q)$ whenever $Q \leq S$ is $\F$-essential. The technique we use will involve the construction of a star of fusion systems $(\T,\F(-),\S(-))$ with $\F_0$ associated to the central vertex and the fusion systems of certain $p'$-extensions of constrained models of the $N_{\F_0}(Q)$ associated to the outer vertices. The saturation of the completion of $(\T,\F(-),\S(-))$ will follow from Theorem \ref{completionsat}.

\begin{Thm}\label{proofsatthm}
Let $\F_0$ be a saturated fusion system on a finite $p$-group $S$. For $1 \leq i \leq m$, let $Q_i \leq S$ be fully $\F_0$-normalised subgroups with $Q_i\varphi \nleq Q_j$ for each $\varphi \in \Hom_{\F_0}(Q_i,S)$ and $i \neq j$. Set $K_i:=\Out_{\F_0}(Q_i)$ and choose $\Delta_i \leq \Out(Q_i)$ so that $K_i$ is a strongly $p$-embedded subgroup of $\Delta_i$. Write $$\F= \langle \{\Hom_{\F_0}(P,S) \mid P \leq S\} \cup \{\Delta_i \mid 1 \leq i \leq m \} \rangle_S.$$ Assume further that for each $1 \leq i \leq m$,
\begin{itemize}
\item[(a)] $Q_i$ is $\F_0$-centric (hence $\F$-centric) and minimal (under inclusion) amongst all $\F$-centric subgroups, and
\item[(b)] no proper subgroup of $Q_i$ is $\F_0$-essential. 
\end{itemize}
Then $\F$ is saturated.
\end{Thm}

\begin{proof}
Since $Q_i$ is fully $\F_0$-normalised and $\F_0$-centric, $N_{\F_0}(Q_i)$ is a constrained fusion system. Hence by Theorem \ref{const} there exists a unique finite group $H_i$ with $Q_i \unlhd H_i$, $C_{H_i}(Q_i) \leq Q_i$ and $N_{\F_0}(Q_i)=\F_{N_S(Q_i)}(H_i)$. Since Out$_{\F_0}(Q_i)=$ Out$_{H_i}(Q_i)$, Lemma \ref{constex} implies that there is some finite group $G_i$ containing $H_i$ with $p \nmid |G_i: H_i|$ and which satisfies Out$_{G_i}(Q_i)=\Delta_i$. Now let $\T$ be the star formed by a central vertex $v_0$ and $m$ vertices $v_i$ each incident to a unique edge $e_i=(v_0,v_i)$ for $1 \leq i \leq m$. Let $(\T,\F(-),\S(-))$ be the tree of fusion systems formed by setting $\F(v_0):=\F_0$, $\F(v_i):=\F_{N_S(Q_i)}(G_i)$ and $\F(e_i):=\F_{N_S(Q_i)}(H_i)=N_{\F_0}(Q_i)$ for $1 \leq i \leq m$. Then $(\T,\F(-),\S(-))$ satisfies $(H)$, $\F_\T=\F$ by Lemma \ref{compft} and it suffices to verify that conditions (a) and (b) of Theorem \ref{completionsat} hold.

Suppose that $P=\S(v_i)$ or $P$ is $\F$-conjugate to an $\F(v_i)$-essential subgroup. We need to show that $P$ is $\F$-centric. In the first case, since $Q_i$ is $\F$-centric (by assumption), $P=N_S(Q_i)$ is also $\F$-centric by Lemma \ref{centlem} (b). In the second case, there exists some $\varphi \in$ Hom$_\F(P,S)$ such that $P\varphi \leq \S(v_i)$ is $\F(v_i)$-essential. If $i \neq 0$ then since $Q_i \unlhd \F(v_i)$, Lemma \ref{esscont} implies that $Q_i \leq P\varphi$ so that since $Q_i$ is $\F$-centric also $P\varphi$ (and hence $P$) is $\F$-centric by Lemma \ref{centlem} (b) again. If $i = 0$ then since $P\varphi$ is not contained in $Q_i$ (by condition (b) in the theorem) for all $1 \leq i \leq m$ and $P\varphi$ is $\F_0$-centric, $P\varphi$ is also $\F$-centric, by the definition of $\F_\T$. 

It remains to prove that (b) holds in Theorem \ref{completionsat}. 
It will be enough to show directly that the connected component of the vertex $[\iota_P^S]_{\F(v_0)}$ in Rep$_{\F_\T}(P,\F(-))$ is a tree for each $\F_\T$-centric subgroup $P \leq S$. Suppose $[\beta]_{\F(v_0)}$ is connected to $[\beta']_{\F(v_i)}$ via an edge $[\alpha]_{\F(e_i)}$ in Rep$_{\F_\T}(P,\F(-))$. If $[\beta']_{\F(v_i)}$ is incident to another edge $[\alpha']_{\F(e_i)}$ then there exists $\rho \in$ Iso$_{\F(v_i)}(P\alpha,P\alpha')$ such that $\alpha'=\alpha \circ \rho$. Since $Q_i \unlhd \F(v_i)$, $\rho$ extends to a map $\bar{\rho} \in$ Hom$_{\F(v_i)}((P\alpha) Q_i,(P\alpha') Q_i)$ with $\bar{\rho}|_{Q_i} \in $ Aut$_{\F(v_i)}(Q_i)$. Set $\chi:=[\rho|_{Q_i}] \in$ Out$_{\F(v_i)}(Q_i)=\Delta_i$ for the class of $\rho|_{Q_i}$ in $\Delta_i$ and note that $\chi \notin K_i$, since otherwise $[\alpha]_{\F(e_i)}=[\alpha']_{\F(e_i)}$. Now $$\chi^{-1} \circ \mbox{Out}_{(P\alpha) Q_i}(Q_i) \circ \chi = \mbox{Out}_{(P\alpha') Q_i}(Q_i) \leq K_i$$ and $$\mbox{Out}_{(P\alpha) Q_i}(Q_i) \leq K_i \cap \chi \circ K_i \circ \chi^{-1},$$ which is a $p'$-group by hypothesis. Since $Q_i$ is $\F$-centric, we have Out$_{(P\alpha) Q_i}(Q_i) \cong (P\alpha) Q_i/Q_i$ so that $P\alpha \leq Q_i$. But $P$ is $\F$-centric and $Q_i$ is minimal amongst all $\F$-centric subgroups (by assumption) and so $P\alpha=Q_i$. Furthermore since $Q_i$ is not $\F_0$-conjugate to $Q_j$ for $i \neq j$, this must occur for some unique $i$. 

Setting $\beta=\iota_P^S$ in the previous paragraph we see that either Rep$_{\F_\T}(P,\F(-))$ is a tree or $P$ is $\F_0$-conjugate to a unique $Q_i$. Suppose we are in this latter situation and let $\Gamma_0 \subset$ Rep$_\F(P,\T)$ be the connected subgraph consisting of all vertices of the form $[\beta]_{\F(v_0)}$ or $[\beta']_{\F(v_i)}$ and edges of the form $[\alpha]_{\F(e_i)}$. If $[\gamma]_{\F(v_j)}$ is a vertex not in $\Gamma_0$ then $j \neq i$ and a minimal path from $[\gamma]_{\F(v_j)}$ to $\Gamma_0$ must consist of a single edge $[\delta]_{\F(e_j)}$, otherwise by the argument in the previous paragraph, we would have that $P$ is $\F_0$-conjugate to $Q_j$, a contradiction (see Figure 4.1 below). Hence $\Gamma_0$ is a deformation retract of  Rep$_{\F_\T}(P,\F(-))$ and it is enough to prove that $\Gamma_0$ is a tree. In fact we prove that $\Gamma_0$ is a star. Suppose that a vertex $[\beta]_{\F(v_0)} \in \Gamma_0$ is incident to two (different) edges $[\alpha]_{\F(e_i)}$ and $[\alpha']_{\F(e_i)}$. The argument in the previous paragraph shows that $\alpha, \alpha' \in$ Iso$_{\F(e_i)}(P,Q_i)$ and there exists $\rho \in$ Aut$_{\F(v_i)}(Q_i)=\Delta_i$ such that $\alpha \circ \rho = \alpha'$. Hence $[\alpha]_{\F(e_i)} = [\alpha']_{\F(e_i)}$ and each $[\alpha]_{\F(v_0)} \in \Gamma_0$ is incident to a unique edge. This establishes the claim, and finishes the proof of the theorem.
\end{proof}

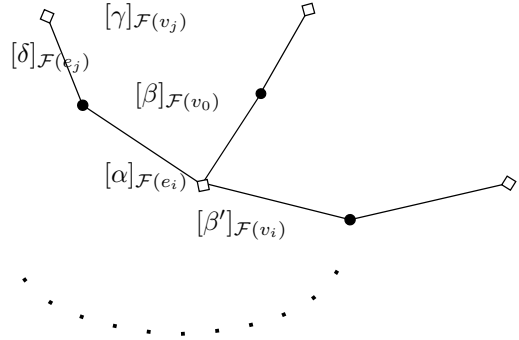
\begin{figure}[!h]\label{treefig}
\centering
\labellist
\pinlabel $[\beta']_{\F(v_i)}$ at 140 70
\pinlabel $[\beta]_{\F(v_0)}$ at 100 150
\pinlabel$[\gamma]_{\F(v_j)}$ at 80 200
\pinlabel$[\alpha]_{\F(e_i)}$ at 80 100
\pinlabel$[\delta]_{\F(e_j)}$ at 20 175
\endlabellist

\includegraphics[scale=0.6]{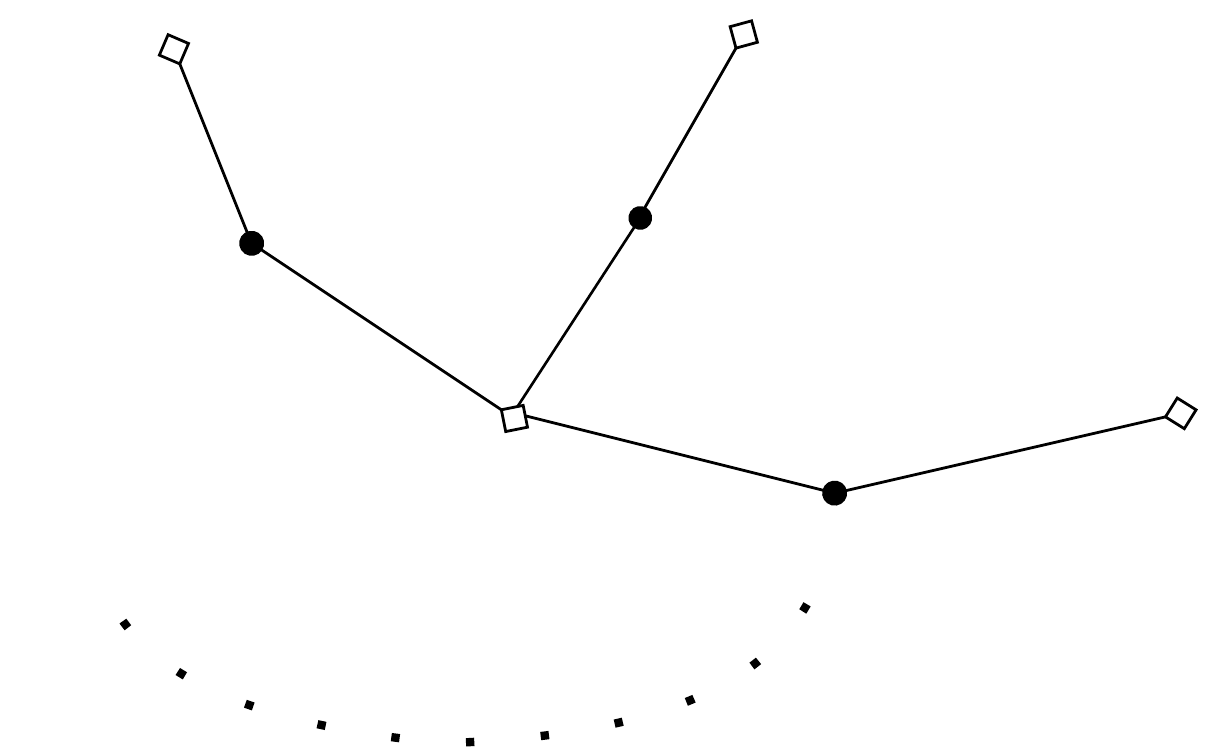}
\caption{Vertices of distance at most two from $[\beta']_{\F(v_i)}$ in $\Gamma$ for each $i \neq j$.} 
\end{figure}

\bibliography{refs}        
\bibliographystyle{plain}
\end{document}